\documentclass[a4paper,10pt,reqno, english]{amsart} 

\usepackage[utf8]{inputenc}
\usepackage{amssymb}
\usepackage{graphicx}
\usepackage{amsmath,bm,amsthm}
\usepackage{amsfonts,amssymb,enumerate}
\usepackage{url,paralist}
\usepackage{mathtools}
\usepackage{float}
\usepackage{dsfont}

\usepackage{caption}
\usepackage{subcaption}
 \usepackage{makecell} 
\usepackage[arrow,curve,matrix,tips,2cell]{xy}  
  \SelectTips{eu}{10} \UseTips
  \UseAllTwocells 

\usepackage[colorlinks=true,urlcolor=blue]{hyperref}
\usepackage{color}
\usepackage{enumerate,amssymb}

\theoremstyle{plain}
\newtheorem{theorem}{Theorem}[section]

\newtheorem{lemma}[theorem]{Lemma}

\newtheorem{corollary}[theorem]{Corollary}
\newtheorem{proposition}[theorem]{Proposition}

\newtheorem*{theorem*}{Theorem}

\newtheorem*{gproblem}{The Gr\"unbaum hyperplane mass partition problem}
\newtheorem*{ghrproblem}{The Gr\"unbaum--Hadwiger--Ramos problem}
\newtheorem*{r-conjecture}{The Ramos conjecture}

\newtheorem*{claim*}{Claim}

\theoremstyle{definition}
\newtheorem{definition}[theorem]{Definition}
\newtheorem{example}[theorem]{Example}

\newcommand\RP{\mathbb{R}{\rm P}}
\newcommand{\R}{\mathbb{R}}

\newcommand{\Z}{\mathbb{Z}}
\newcommand{\F}{\mathbb{F}}

\newcommand{\E}{\mathrm{E}}

\newcommand\Sym{\mathfrak{S}}
\newcommand\Wk{\Sym^\pm}

\newcommand\tildephi{\overline{\phi}}
\newcommand\tildepsi{\overline{\psi}}

\newcommand{\im}{\operatorname{im}}

\newcommand{\ind}{\operatorname{Index}}

\begin{document}

\title[Topology of the Gr\"unbaum--Hadwiger--Ramos problem]
{Topology of the Gr\"unbaum--Hadwiger--Ramos hyperplane mass partition problem}

\author[Blagojevi\'c]{Pavle V. M. Blagojevi\'{c}} 
\thanks{The research by Pavle V. M. Blagojevi\'{c} leading to these results has
        received funding from the Leibniz Award of Wolfgang L\"uck granted by DFG.
        Also supported by the grant ON 174008 of the Serbian Ministry of Education and Science.}
\address{Inst. Math., FU Berlin, Arnimallee 2, 14195 Berlin, Germany\hfill\break
\mbox{\hspace{4mm}}Mat. Institut SANU, Knez Mihailova 36, 11001 Beograd, Serbia}
\email{blagojevic@math.fu-berlin.de} 
\author[Frick]{Florian Frick}
\thanks{The research of Florian Frick and of Albert Haase leading to these results has
        received funding from German Science Foundation DFG via the Berlin Mathematical School.}
\address{Dept.\ Math., Cornell University, Ithaca, NY 14853, USA}
\email{ff238@cornell.edu}
\author[Haase]{Albert Haase}
\address{Inst. Math., FU Berlin, Arnimallee 2, 14195 Berlin, Germany} 
\email{albert.haase@gmail.com}
\author[Ziegler]{G\"unter M. Ziegler} 
\thanks{The research by G\"unter M. Ziegler leading to these results has received funding from the European Research Council under the European Union's Seventh Framework Programme (FP7/2007-2013) / ERC Grant agreement no.~247029-SDModels
	 	and from the DFG Collaborative Research Center TRR~109 ``Discretization in Geometry and Dynamics''.}  
\address{Inst. Math., FU Berlin, Arnimallee 2, 14195 Berlin, Germany} 
\email{ziegler@math.fu-berlin.de}

\begin{abstract}
  In 1960 Grünbaum asked whether for any finite mass in $\R^d$
  there are $d$ hyperplanes that cut it into $2^d$ equal parts.
  This was proved by Hadwiger (1966) for $d\le3$, but disproved
  by Avis (1984) for $d\ge5$, while the case $d=4$ remained open.
  
  More generally, Ramos (1996) asked for
  the smallest dimension $\Delta(j,k)$ in which for any $j$ masses
  there are $k$ affine hyperplanes that simultaneously cut each of 
  the masses into $2^k$ equal parts.
  At present the best lower bounds on $\Delta(j,k)$
  are provided by Avis (1984) and Ramos (1996), the best upper bounds 
  by Mani-Levitska, Vre\'cica \& \v{Z}ivaljevi\'c (2006).
  The problem has been an active testing ground for  
  advanced machinery from equivariant topology.
  
  We give a critical review of the work on the Grünbaum--Hadwiger--Ramos problem, 
  which includes the documentation of essential gaps in the proofs for some previous claims.
  Furthermore, we establish that $\Delta(j,2)= \frac12(3j+1)$ in the cases
  when $j-1$ is a power of~$2$, $j\ge5$. 
\end{abstract}



\date{February 10, 2015; revised February 8, 2018}

\maketitle


\section{Introduction and statement of main results}
\label{sec : Introduction and statement of main results}

In 1960 Branko Grünbaum \cite[Sec.\,4.(v)]{grunbaum1960partitions} suggested the following innocent-looking problem:
\begin{gproblem}
    Can any convex body in $\R^d$ be cut into $2^d$ pieces of equal volume
    by $d$ suitably-chosen affine hyperplanes?
\end{gproblem}
As Grünbaum noted, this is quite easy to prove for $d\le2$.
In 1966 Hadwiger \cite{hadwiger1966} answered Grünbaum's question (positively)
for $d=3$, while solving a problem raised by J. W. Jaworowski (Oberwolfach, 1963).
In the course of his argument, he was led to consider
the partition of two masses by two planes in $\R^3$.

Grünbaum's question was independently raised in computational geometry,
motivated by the construction of data structures for range queries.
In this context, Willard \cite{Willard:polygon} reproved the case $d=2$, while
the case $d=3$ was reproved by 
Yao, Dobkin, Edelsbrunner \& Paterson \cite{YaoDobkinEdelsbrunnerPaterson}. 

In 1984 Avis \cite{avis1984} answered Grünbaum's problem negatively for $d\ge5$.
Indeed, one cannot expect a positive answer there, since 
$d$ hyperplanes in $\R^d$ can be described by $d^2$ parameters,
while the hyperplanes one is looking for need to satisfy $2^d-1$
independent conditions, and $2^d-1>d^2$ for $d>4$. 
The case $d=4$ was left open: 
“The situation in $4$-space is not settled in general” \cite[p.~125]{avis1984}.

In 1996 Ramos \cite{ramos1996equipartition} formulated the general version of the hyperplane mass partition
problem for several masses:

\begin{ghrproblem}
\label{GruenbaumProblem0}
For each $j\ge1$ and $k\ge1$, determine the smallest dimension $d=\Delta (j,k)$ such that for every collection of $j$ masses
$\mathcal{M}$ on $\R^d$ there are $k$ affine hyperplanes that cut each of the $j$ masses into $2^k$ equal pieces.
\end{ghrproblem}

So Grünbaum's question was whether $\Delta(1,k)\le k$.
The special case $\Delta(j,1)=j$ of the Grünbaum--Hadwiger--Ramos problem, 
for a single hyperplane $(k=1)$,
is settled by the so-called ham-sandwich theorem,  which was
conjectured by Steinhaus and proved by Banach in 1938 (see \cite{BeyerZardecki}).
This turned out to be an incarnation of the Borsuk--Ulam theorem, 
an important and central result of early algebraic topology, with
applications that range from discrete geometry to nonlinear PDEs; 
see \cite{Steinlein1,Steinlein2} for surveys. 

It turns out, however, that the most natural configuration spaces parameterizing
$k$-tuples of oriented affine hyperplanes are products of spheres, such as $(S^d)^k$, which do 
not have the high connectivity that is required for a simple application of Borsuk--Ulam
type machinery (e.g.\ via Dold's Theorem; see Matou\v{s}ek \cite{Matousek-BU} for an introduction
to this approach).
Thus more sophisticated machinery is needed in order to decide about the 
existence of the equivariant maps proposed by various applications of 
the ``Configuration Space/Test Map Scheme'' (as developed by Sarkaria and \v{Z}ivaljevi\'c; see
again \cite{Matousek-BU} for an introduction).  
Methods that have been employed to settle such existence problems include
\begin{compactitem}[ --]
        \item equivariant cohomology (the Fadell--Husseini index \cite{husseini1988}),
        \item equivariant obstruction theory (see tom Dieck \cite[Sec.\,II.3]{tom1987transformation}), and
        \item the normal bordism approach of Koschorke \cite{koschorke1981vector}.
\end{compactitem}
In this paper we attempt to provide a status report about the partial results
obtained for the Grünbaum--Hadwiger--Ramos problem up to now. This in particular 
includes the lower and upper bounds  
\[
      \big\lceil\tfrac{2^k-1}{k}j\big\rceil\ \le\ \Delta(j,k)\ \le\ j + (2^{k-1}-1)2^{\lfloor\log_2j\rfloor},
\]
where $2^{\lfloor\log_2j\rfloor}$ is $j$ “rounded down to the nearest power of $2$,”
so $\frac12j<2^{\lfloor\log_2j\rfloor}\le j$.

The lower bound was derived by Avis \cite{avis1984} (for $j=1$) and Ramos \cite{ramos1996equipartition}
from measures concentrated on the moment curve.
The upper bound was obtained by Mani-Levitska, Vre\'cica \& \v{Z}ivaljevi\'c \cite{mani2006}
from a Fadell--Husseini index calculation. 
A table below will show that there is quite a gap between the lower and
the upper bounds --- they only coincide in the ham-sandwich case $\Delta(j,1)=j$, 
and in the case of
two hyperplanes if $j+1$ is a power of~$2$,  with $\Delta(j,2)=\frac12(3j+1)$.
All the available evidence, up to now is consistent with the expectation that
Ramos' lower bound is tight for all $j$ and~$k$; we will refer to this in the following
as the \emph{Ramos conjecture}. For example, while the above bounds specialize to
$3\le\Delta(2,2)\le4$,  Hadwiger \cite{hadwiger1966} proved  
that indeed $\Delta(2,2)=3$.
 
In addition to the general lower and upper bounds, a number of papers have treated special
cases, reductions, and relatives of the problem.
As a basis for further work (by the present authors and by others), we will in the following
provide a critical review of all the key contributions to this study, which will also include
short proofs as far as feasible. In this context we have to observe, however, that quite a number
of published proofs do not hold up upon critical inspection, and indeed 
some of the approaches employed cannot work.
As some of these errors have not been pointed out in print (although they may be known to experts),
we will provide detailed reviews and explanations in these cases.

We have, however, been able to salvage one of these results, with different methods:
We will prove below (Theorem~\ref{thm:Delta(2^t+1,2)}) that
$\Delta(j,2)= \frac12(3j+1)$ also holds if $j-1$ is a power of~$2$, $j\ge5$. 
So in this case again the
Ramos lower bound is tight while the Mani-Levitska et al.\ 
upper bound is not. (It is tight in the case $j=3$.)

\subsection{Set-up and terminology}
\label{sec : Set-up and terminology}

Any affine hyperplane $H=H_{v}(a)=\{x\in\R^d : \langle x,v\rangle =a\}$, 
given by a vector $v\in\R^d{\setminus}\{0\}$ and scalar $a\in\R$, 
determines two closed halfspaces, which we denote by 
\[
H^{0}=\{x\in \R^d :\langle x,v\rangle \geq a\}\qquad\text{and}\qquad
H^{1}=\{x\in\R^d : \langle x,v\rangle \leq a\}.
\]
Let $\mathcal{H}$ be an \emph{arrangement} (ordered tuple) of $k\ge1$ affine hyperplanes in $\R^d$, and $\alpha=(\alpha_1,\ldots,\alpha_k)\in (\Z/2)^k=\{0,1\}^k$.
The \emph{orthant} determined by the arrangement $\mathcal{H}$ and an element $\alpha\in(\Z/2)^k$ is the intersection of halfspaces 
\[
\mathcal{O}_{\alpha}^{\mathcal{H}}=H_1^{\alpha_{1}}\cap\cdots\cap H_k^{\alpha_{k}}.
\]
    A \emph{mass} on $\R^d$ is a finite Borel measure on $\R^d$
    that vanishes on every affine hyperplane. 
Without loss of generality we deal only with probability measures
(that is, masses such that $\mu(\R^d)=1$).
Examples of masses that appear frequently include
\begin{compactitem}[ --]
    \item measures given by the $d$-dimensional volume of a compact convex body $K\subset\R^d$,
    \item measures induced by an interval on the moment curve in $\R^d$,
    \item measures given by a finite family of (small, disjoint) balls.
\end{compactitem}
An arrangement $\mathcal{H}=(H_1,\ldots,H_k)$ \emph{equiparts} a collection of
masses $\mathcal{M}=(\mu_1,\ldots,\mu_j)$ if for every element $\alpha\in(\Z/2)^k$ and every $\ell\in\{1,\ldots,j\}$
\[
\mu_{\ell} (\mathcal{O}_{\alpha}^{\mathcal{H}})=\tfrac{1}{2^{k}}.
\]
Clearly this can happen only if $k\le d$. 

The Grünbaum--Hadwiger--Ramos problem thus 
asks for the smallest dimension $d=\Delta(j,k)$ 
in which any collection $\mathcal{M}$ of $j$~masses in $\R^d$ admits an 
arrangement $\mathcal{H}$ of $k$ affine hyperplanes that equi\-parts~$\mathcal{M}$.

For the proofs using equivariant topology methods,
we make additional assumptions on the masses to be considered, namely
that the measures $\mu_i$ that we deal with have compact connected support.
This assumption can be made as we can strongly approximate each mass 
by masses with compact connected support.
(This can be done ``\emph{mit passender Grenzbetrachtung und Kompaktheitserwägung auf die übliche schulmäßige Weise}'' \cite[S.~275]{hadwiger1966} as we learn from Hadwiger.)
It guarantees that the measure captured by an affine halfspace 
depends continuously on the halfspace, and more generally that the measure captured
by an orthant depends continuously on the hyperplanes that define the orthant.
Moreover, it yields that for any mass $\mu$ and 
a given vector $v$ 
the hyperplane $H_{v}(a)$ that halves the mass $\mu$ is unique, and 
depends continuously on~$v$.

One could also allow for measures supported on finitely many points,
as often considered in the computational geometry context; see e.g.\ 
\cite{avis1984} and \cite{YaoDobkinEdelsbrunnerPaterson}. 
Such point measures 
do \emph{not} satisfy the assumptions above, but they can be approximated by masses that do.
To accomodate for point measures, one would have to modify the definition of ``equiparts'' 
in such a way that
each open orthant captures \emph{at most} a fraction of $1/2^k$ of each measure.

\subsection{Summary of known Results}
\label{sec: Known Results}

We have noted that the ham-sandwich theorem yields $\Delta(j,1)=j$ and that trivially  $k\le\Delta(j,k)$.
A stronger lower bound 
was given by Ramos \cite{ramos1996equipartition}:
\begin{equation}
    \label{ramos_bound}
    \tfrac{2^k-1}{k}j\leq\Delta(j,k).
\end{equation}
Ramos believed that his bound is tight:

\begin{r-conjecture} 
$\Delta(j,k)=\lceil \tfrac{2^k-1}{k}j\rceil$ for every $j\geq 1$ and $k\geq 1$.
\end{r-conjecture}

The best upper bound to date, 
due to Mani-Levitska et al.~\cite[Thm.\,39]{mani2006}, can be phrased as follows:
\begin{equation}
    \label{mani_bound}
    \Delta(2^t+r,k)\leq 2^{t+k-1}+r\qquad\textrm{ for }t\ge0,\ 0\leq r\leq 2^t-1.
\end{equation}
The proofs of these bounds are subject of Section~\ref{sec:bounds} (Theorems~\ref{thm:lower-bound} and \ref{thm:upper-bound}).
In particular, for $k=2$ and $j=2^{t+1}-1$ the lower bound \eqref{ramos_bound} and 
the upper bound~\eqref{mani_bound} coincide, implying that
\begin{equation*}
\Delta(2^{t+1}-1,2)=3\cdot2^t-1\qquad\textrm{ for }t\ge0.
\end{equation*}

The first result that is not a consequence of a coincidence between the lower and upper bounds
\eqref{ramos_bound} and \eqref{mani_bound} is due to Hadwiger \cite{hadwiger1966},
who showed that two masses in~$\R^3$ can be simultaneously cut into four equal parts by two (hyper)planes.
We give a degree-based proof for this result in Section \ref{sec:hadwiger} (Theorem~\ref{thm:hadwiger}):
\begin{equation*}
\Delta(2,2)=3.
\end{equation*} 
As Hadwiger observed, by a simple reduction \eqref{eq:reduction_Ramos}  
this also implies that $\Delta(1,3)=3$.

Despite a number of published papers in prominent journals on 
new cases of the Ramos conjecture, the values and bounds for $\Delta(j,k)$ 
just mentioned appear to be the only ones available before with correct proofs:
The papers 
by Ramos \cite{ramos1996equipartition} from 1996, 
by Mani-Levitska et al.~\cite{mani2006} from 2006, 
and by \v{Z}ivaljevi\'c \cite{zivaljevic2008} from 2008 and  
\cite{zivaljevic2011equipartitions} from 2011 (published version \cite{zivaljevic2011equipartitions-published} from 2015) 
all contain essential gaps; 
see Sections~\ref{sec : the failure}, \ref{sec:ramos} and \ref{sec:further_gaps}.
In Table~\ref{table:claimed_bounds} we summarize the situation. 

\begin{table}
{\begin{tabular}{r c c c l | l}
{Lower} & & $\Delta(j,k)$ & & {Upper}  & {Reference for upper bound} \\
\hline 
$8$ & $\le$ & $\Delta(5,2)$ & $\le$ & ${\color{red}8}$  & \cite[Thm.\,4]{mani2006}
\\[1.7pt]
$\tfrac32 \cdot 2^{t}$ & $\le$ & $\Delta(2^{t},2)$ & $\le$ & ${\color{red}\tfrac32\cdot 2^t}$  & \cite[Thm.\,6.3]{ramos1996equipartition} \cite[Prop.\,25]{mani2006} 
\\[1.7pt]
$\tfrac32 \cdot 2^{t} +2$ & $\le$ & $\Delta(2^{t}+ 1,2)$ & $\le$ & ${\color{red}\tfrac32\cdot 2^t +2}$ &\cite[Thm.\,2.1]{zivaljevic2011equipartitions}~\cite[Thm.\,2.1]{zivaljevic2011equipartitions-published}
\\[1.7pt]
$\tfrac{7}{3} \cdot 2^{t}$ & $\le$ & $\Delta(2^{t},3)$ & $\le$ & ${\color{red}\tfrac52\cdot 2^t}$  &
 \cite[Thm.\,6.3]{ramos1996equipartition} 
\\[1.7pt]
$4$ & $\le$ & $\Delta(1,4)$ & $\le$ & ${\color{red}5}$  & \cite[Thm.\,6.3]{ramos1996equipartition} 
\\[1.7pt]
$\tfrac{15}{4} \cdot 2^{t}$ & $\le$ & $\Delta(2^{t},4)$ & $\le$ & ${\color{red}\tfrac92\cdot 2^t}$  &
 \cite[Thm.\,6.3]{ramos1996equipartition} 
\\[1.7pt]
$7$ & $\le$ & $\Delta(1,5)$ & $\le$ & ${\color{red}9}$  & \cite[Thm.\,6.3]{ramos1996equipartition} 
\\[1.7pt]
$\tfrac{31}{5} \cdot 2^{t}$ & $\le$ & $\Delta(2^{t},5)$ & $\le$ & ${\color{red}\tfrac{15}2\cdot 2^t}$  &
 \cite[Thm.\,6.3]{ramos1996equipartition} \\[1.7pt] \hline
\end{tabular}
}
\caption{Various upper bounds claimed in the literature with incorrect/incomplete
proofs, where $t\ge1$. 
For comparison, we also show the Ramos lower bounds, which are conjectured to be tight.}
\label{table:claimed_bounds}
\end{table}

\smallskip\noindent
Furthermore, in Section~\ref{sec : the failure} we show that 
\v{Z}ivaljevi\'c's approach in \cite{zivaljevic2008} towards  
the last remaining open case $\Delta(1,4)=4$ of the Grünbaum problem
fails in principle as well as in details. 

Finally, in Section~\ref{sec : D(2t+1,2)}  we prove using a degree calculation that 
\begin{equation}
\label{values of Delta(j,k) : No 3}
\Delta(2^t+1,2)=3\cdot 2^{t-1}+2\quad\text{for }t\geq 2.
\end{equation} 
By this we verify an instance of the Ramos conjecture previously claimed  by 
\v{Z}ivaljevi\'c in \cite[Thm.\,2.1]{zivaljevic2011equipartitions}~\cite[Thm.\,2.1]{zivaljevic2011equipartitions-published}.

The resulting status of the Grünbaum--Hadwiger--Ramos problem
is summarized in Table~\ref{table:summary}.

\begin{table}[ht] 
\begin{tabular}{|c|| p{2.2cm} | p{2.2cm} |  p{2.2cm}  |  p{2.2cm} |} \hline
\multicolumn{5}{|c|}{Values of $\Delta(j,k)$ for $j$ measures and $k$ hyperplanes and $t\geq 1$} \\ \hline
\diaghead{ \tiny xxxxxxxxxxxx}{ $j$  }{ $k$ } & \centering$1$ & \centering$2$ & \centering$3$ & \centering$4$\arraybackslash\\ \hline\hline
  & \tiny$1\le$\hfill&\tiny$2\le$\hfill&\tiny$3\le$\hfill&\tiny$4\le$\hfill \\[1mm]
1 & \centering$1$ & \centering$2$ & \centering$\boldsymbol{3}$ &   \\[1mm] 
  & \hfill\tiny$\le1$&\hfill\tiny$\le2$&\hfill\tiny$\le4$&\hfill\tiny$\le8$ \\ \hline
  & \tiny$2\le$\hfill&\tiny$3\le$\hfill&\tiny$5\le$\hfill&\tiny$8\le$\hfill \\[1mm]
2 & \centering$2$ & \centering$\boldsymbol{3}$ &   &   \\[1mm] 
  & \hfill\tiny$\le2$&\hfill\tiny$\le4$&\hfill\tiny$\le8$&\hfill\tiny$\le16$ \\ \hline
  & \tiny$3\le$\hfill&\tiny$5\le$\hfill&\tiny$7\le$\hfill&\tiny$12\le$\hfill \\[1mm]
3 & \centering$3$ & \centering$5$ &   &   \\[1mm] 
  & \hfill\tiny$\le3$&\hfill\tiny$\le5$&\hfill\tiny$\le9$&\hfill\tiny$\le17$ \\ \hline
$\vdots$ & \centering$\vdots$ &  &  &  \\ \hline
  & \tiny$2^{t} -1\le$\hfill&\tiny$3\cdot2^{t-1}-1\le$\hfill&\tiny$ $\hfill&  \\[1mm]
$2^{t} -1$ & \centering$2^{t} -1$ & \centering$3 \cdot 2^{t-1} -1$ &   &   \\[1mm] 
  & \hfill\tiny$\le2^{t} -1$&\hfill\tiny$\le3\cdot2^{t-1}-1$&\hfill\tiny$ $&  \\ \hline
  & \tiny$2^{t} \le$\hfill&\tiny$3\cdot2^{t-1}\le$\hfill&\tiny$ $\hfill&  \\[1mm]
$2^{t}  $ & \centering$2^{t}$ & \centering$\boldsymbol{\le 3 \cdot 2^{t-1} +1}$ &   &   \\[1mm] 
  & \hfill\tiny$\le2^{t} $&\hfill\tiny$\le4\cdot2^{t-1}$&\hfill\tiny$ $&  \\ \hline
  & \tiny$2^{t} +1\le$\hfill&\tiny$3\cdot2^{t-1}+2\le$\hfill&\tiny$ $\hfill&  \\[1mm]
$2^{t} +1$ & \centering$2^{t} +1$ & \centering$\boldsymbol{3 \cdot 2^{t-1} +2}$ &   &   \\[1mm] 
  & \hfill\tiny$\le2^{t} +1$&\hfill\tiny$\le4\cdot2^{t-1}+1$&\tiny\hfill$ $&  \\ \hline
\end{tabular}
\caption{Each square in this table records the 
lower bound \eqref{ramos_bound} in the north-west corner, the
upper bound \eqref{mani_bound} in the south-east corner, 
and the exact value or improved bound 
in the center. 
The values/bounds that do \emph{not} simply 
follow from the two bounds coinciding are typeset in \textbf{boldface}.}
\label{table:summary}
\end{table}

\section{Transition to equivariant topology}
\label{sec:conf}

In this section we demonstrate how the Gr\"{u}nbaum--Hadwiger--Ramos problem induces a problem of Borsuk--Ulam type.

Consider a collection of $j$ masses $\mathcal{M}=(\mu_1, \dots, \mu_j)$ on $\R^d$.
We would like to find an arrangement of $k$ affine hyperplanes $\mathcal{H}=(H_1, \dots, H_k)$ in $\R^d$ such that 
$\mathcal{H}$ equiparts $\mathcal{M}$.

\subsection{The configuration spaces}

The sphere $S^d$ can be seen as the space of all oriented affine hyperplanes 
in~$\R^d$, where the north pole $e_{d+1}$ and the south pole $-e_{d+1}$ lead to \emph{hyperplanes at infinity}. 
For this we embed $\R^d$ into $\R^{d+1}$ via the map $(x_1,\dots, x_d) \longmapsto(x_1,\ldots,x_d,1)$. 
An oriented affine hyperplane in $\R^d$ is mapped to an oriented affine $(d-1)$-dimensional subspace of $\R^{d+1}$ and is extended (uniquely) to an oriented linear hyperplane. 
The unit normal vector on the positive side of the linear hyperplane defines a point on the sphere $S^d$.  
There is a one-to-one correspondence between points $v$ in $S^d \setminus \{e_{d+1},-e_{d+1}\}$ and oriented affine hyperplanes $H_v$ in $\R^d$.
Let $H_v^{0}$ and $H_v^{1}$ denote the positive resp.\ the negative closed half-space determined by $H_v$. 
The positive side of the hyperplane at infinity
is $\R^d$ for $v=e_d$ and $\emptyset$ for $v=- e_d$.
Hence $H_{-v}^{0} = H_{v}^{1}$ for every~$v$.

There are three natural configuration spaces that parametrize arrangements of $k$ oriented affine hyperplanes in $\R^d$. 
Note that hyperplanes at infinity cannot arise as solutions to the mass partition problem, since they produce empty orthants. 
Hence we do not need to worry about the fact that the following configuration spaces incorporate these.
 
The configuration spaces we consider are 
\begin{enumerate}[\normalfont (i)] \label{enum_three_cstm_schemes}
\item the \emph{join configuration space} $X_{d,k} = (S^d)^{*k} \cong S^{dk + k-1}$, the $k$-fold join of spheres $S^d$,
\item the \emph{product configuration space} $Y_{d,k} = (S^d)^k$, the $k$-fold Cartesian product of spheres $S^d$, and
\item the \emph{free configuration space} 
$Z_{d,k} = \{(x_1,\dots,x_k)\in Y_{d,k}: x_i\neq\pm x_j\text{\,for\,}i<j\}$, the largest subspace of $Y_{d,k}$ on
which the group action described below is free. 
\end{enumerate}

\subsection{The group}
The Weyl group $\Wk_k=(\Z/2)^k \rtimes \Sym_k$, also known as the group of signed permutations, or 
as the symmetry group of the $k$-dimensional cube, acts naturally  
on the configuration spaces we consider:
It permutes the hyperplanes, and changes their orientations. 
Correspondingly it also acts on the test spaces, which record the fractions of the $j$ masses captured in each of the $2^k$ orthants.

\subsection{The action on configuration spaces}
\label{sub:action of configuration space}
Elements in $X_{d,k}$ can be presented as formal ordered convex combinations $t_1 v_1 + \dots + t_k v_k$, where $t_i \geq 0,\,\sum t_i = 1$ and $v_i \in S^d$. 
The action of the group $\Wk_k = (\Z/2)^k \rtimes \Sym_k$ on the space $X_{d,k}$ is defined as follows. 
Each copy of $\Z/2$ acts antipodally on the corresponding sphere $S^d$ while the symmetric group $\Sym_k$ acts by permuting coordinates. 
More precisely, let $((\beta_1, \dots, \beta_k) \rtimes \tau) \in \Wk_k$ and $t_1 v_1 + \dots + t_k v_k \in X_{d,k}$, then
\begin{multline*}
((\beta_1, \dots, \beta_k) \rtimes \tau)\cdot (t_1 v_1 + \dots + t_k v_k) =\\
t_{\tau^{-1}(1)} (-1)^{\beta_1} v_{\tau^{-1}(1)} + \dots +  t_{\tau^{-1}(k)} (-1)^{\beta_k} v_{\tau^{-1}(k)}.
\end{multline*} 
The diagonal subspace 
$\{ \tfrac1kv_1+\dots+\tfrac1kv_k \in X_{d,k} \} \cong Y_{d,k}$ of $X_{d,k}$ 
is invariant under the $\Wk_k$-action and thus has a well-defined induced $\Wk_k$-action. Furthermore, there is a well-defined induced action of $\Wk_k$ on $Z_{d,k}$, since the action leaves the subset $Y_{d,k}^{>1}$ of
all points in~$Y_{d,k}$ with non-trivial stabilizers invariant. 
Note that for $k \ge 2$ the $\Wk_k$-action is free on~$Z_{d,k}$ but not on $X_{d,k}$ or on $Y_{d,k}$.

\subsection{The test space}
\label{sub:test space}
Consider the vector space $\R^{(\Z/2)^k}$ and the subspace of codimension $1$
\[
U_k=\Big\{ (y_{\alpha})_{\alpha\in (\Z/2)^k} \in \R^{(\Z/2)^k} : \sum_{\alpha\in (\Z/2)^k} y_{\alpha} =0\Big\}.
\]
We define an action of $\Wk_k$ on $\R^{(\Z/2)^k}$ as follows:
$((\beta_1, \dots, \beta_k) \rtimes \tau) \in \Wk_k$  acts on a vector $(y_{(\alpha_1,\dots,\alpha_k)})_{(\alpha_1,\dots,\alpha_k) \in (\Z/2)^k}\in \R^{(\Z/2)^k}$ by acting on its indices
\[
((\beta_1, \dots, \beta_k) \rtimes \tau) \cdot (\alpha_1, \dots, \alpha_k) = (\beta_1 + \alpha_{\tau^{-1}(1)}, \dots, \beta_k + \alpha_{\tau^{-1}(k)}),
\]
where the addition is in $\Z/2$. 
With respect to this action of $\Wk_k$ the subspace $U_k$ is an $\Wk_k$-subrepresentation. 

The \emph{test space} related to both configuration spaces $Y_{d,k}$ and $Z_{d,k}$ and a family of $j$ masses is the $\Wk_k$-representation $U_k^{\oplus j}$, where the action is diagonal.

\subsection{The test map}

Consider the following map from the configuration space $Y_{d,k}$ to the test space $U_{k}^{\oplus j}$ associated to the collection of masses $\mathcal{M}=(\mu_1, \dots, \mu_j)$:
\begin{align*}
\phi_{\mathcal{M}} \colon Y_{d,k} &\longrightarrow U_{k}^{\oplus j},\\
(v_1, \dots, v_k) &\longmapsto \Big( \big(\mu_i(H^{\alpha_1}_{v_1} \cap \dots \cap H^{\alpha_k}_{v_k}) - \tfrac{1}{2^k} \big)_{(\alpha_1,\dots,\alpha_k)\in (\Z/2)^k} \Big)_{i \in \{1,\dots, j\}}.
\end{align*}
The map $\phi_{\mathcal{M}}$ is $\Wk_k$-equivariant with respect to the actions introduced in Sections~\ref{sub:action of configuration space} and \ref{sub:test space}.
The essential property of the map $\phi_{\mathcal{M}}$ is that \emph{the oriented hyperplanes $H_{v_1}, \dots, H_{v_k}$ equipart $\mathcal{M}$ if and only if $\phi_{\mathcal{M}} (v_1, \dots, v_k) = 0\in U_{k}^{\oplus j}$}. 
Note that the space $U_{k}^{\oplus j}$ does not depend on the dimension $d$.

Finally, we define the $\Wk_k$-equivariant map $\psi_{\mathcal{M}} \colon Z_{d,k}\longrightarrow U_{k}^{\oplus j}$ as the restriction of $\phi_{\mathcal{M}}$ to $Z_{d,k}$.
Again, the essential property holds: \emph{The oriented hyperplanes $H_{v_1}, \dots, H_{v_k}$ equipart $\mathcal{M}$ if and only if $\psi_{\mathcal{M}} (v_1, \dots, v_k) = 0 \in U_{k}^{\oplus j}$}. 

The maps $\phi_{\mathcal{M}} $ and $\psi_{\mathcal{M}}$ are called \emph{test maps}.
Thus we have established the following criteria.

\begin{proposition}
\label{prop:CS/TM}
Let $d\ge1$, $k\ge1$, and $j\ge1$.
\begin{compactenum}[\rm (1)]
\item Let $\mathcal{M}$ be a collection of $j$ masses on $\R^d$, and let $\phi_{\mathcal{M}} \colon Y_{d,k} \longrightarrow U_{k}^{\oplus j}$ and $\psi_{\mathcal{M}} \colon Z_{d,k} \longrightarrow U_{k}^{\oplus j}$ be the $\Wk_k$-equivariant maps defined above.
If \,$0\in \im\phi_{\mathcal{M}}$, or\, $0\in \im\psi_{\mathcal{M}}$, then there are $k$ oriented hyperplanes that equipart $\mathcal{M}$.
\item Let $S(U_{k}^{\oplus j})$ denote the unit sphere in the vector space $U_{k}^{\oplus j}$.
If there is no $\Wk_k$-equivariant map $Y_{d,k} \longrightarrow S(U_{k}^{\oplus j})$, or $Z_{d,k} \longrightarrow S(U_{k}^{\oplus j})$, then $\Delta(j,k)\le d$.
\end{compactenum}
\end{proposition}

\noindent
We have an equivalence $0\in \im\phi_{\mathcal{M}} \Longleftrightarrow 0\in \im\psi_{\mathcal{M}}$, since
on the non-free part two hyperplanes are equal or opposite, so some orthants are empty, and we do
not loose any equipartitions by deleting the non-free part.
However, the non-existence of an $\Wk_k$-equivariant map $Z_{d,k} \longrightarrow S(U_{k}^{\oplus j})$ only implies the non-existence of an $\Wk_k$-equivariant map $Y_{d,k} \longrightarrow S(U_{k}^{\oplus j})$, but not conversely.

The join configuration spaces $X_{d,k}$ were introduced in \cite{blagojevic2011}.
They will not be used here, but will be essential in our subsequent work \cite{BlagojevicFrickHaaseZiegler2}. 
The construction of the corresponding $\Wk_k$-equivariant test map is given in \cite[Sec.\,2.1]{blagojevic2011}.
The product configuration spaces $Y_{d,k}$ embeds into the join configuration space $X_{d,k}$ via the diagonal embeddings $Y_{d,k}\lhook\joinrel\longrightarrow X_{d,k}$, $(v_1,\dots,v_k) \longmapsto \frac{1}{k}v_1 + \dots + \frac{1}{k}v_k$. 
They play a central role for the configuration space/test map scheme that will produce all major results in the following.
The free configuration spaces $Z_{d,k}$ appear in the literature as \emph{orbit configuration spaces}; see for example \cite{feichtner},
where they are denoted by $F_{\Z/2}(S^d,k)$. 
We will show below that the restriction of the configuration space/test map scheme 
to $Z_{d,k}$ is problematic, as
for this restricted scheme the equivariant maps, 
whose non-existence would be needed for settling new cases of the Ramos conjecture,
do exist, partially for trivial reasons; see in particular Section~\ref{sec : the failure}.%

\section{Bounds and reductions for $\Delta(j,k)$}
\label{sec:bounds}

In this section we present the general lower and upper bounds for the function $\Delta(j,k)$.
For the sake of completeness we present proofs. 

\subsection{The lower bounds by Ramos}
\label{subsec:lower-bounds}

\begin{theorem}[Ramos {\cite{ramos1996equipartition}}]\label{thm:lower-bound}
For $j\ge1$ and $k\ge1$, the minimal dimension $d = \Delta(j,k)$ such that any $j$ masses on $\R^d$ can be equiparted by $k$ hyperplanes satisfies
\[
\lceil \tfrac{(2^k -1)}{k}j \rceil \le \Delta(j,k).
\]
\end{theorem}

\begin{proof}
Let $\gamma\colon\R \longrightarrow \R^d$ given by $\gamma(t) = (t,t^2,\dots, t^d)$ be the moment curve in~$\R^d$.  
Choose $j$ pairwise disjoint intervals on this curve and let $\mu_1,\dots,\mu_j$
be the corresponding masses.
Any equipartition of these masses by $k$ hyperplanes must give rise to at least $(2^k-1)j$ intersections of the hyperplanes with $\im\gamma$. 
The result now follows if we recall that the moment curve has degree $d$:
Any hyperplane meets it in at most $d$ distinct
points, so  $k$ hyperplanes can intersect it in at most $dk$ points.
\end{proof}

\subsection{The upper bounds by Mani-Levitska et al.}
\label{subsec:upper-bounds}
\begin{theorem}[Mani-Levitska et al.~{\cite[Thm.\,39]{mani2006}}]\label{thm:upper-bound}
Given integers $t \geq 0$, $k \geq 1$, and $0\leq r\leq 2^t-1$, the minimal dimension $d = \Delta(2^t+r,k)$ such that any $j = 2^t + r$ masses on $\R^d$ can be equiparted by $k$ hyperplanes satisfies
\[
\Delta(2^t + r,k) \le 2^{t + k - 1} + r.
\]
\end{theorem}
\begin{proof}
Let $d=2^{t + k - 1} + r$ and $j= 2^t + r$. 
According to Proposition~\ref{prop:CS/TM} it suffices to prove that there is 
no $(\Z/2)^k$-equivariant, and consequently no $\Wk_k$-equivariant, map  
$Y_{d,k}\longrightarrow S(U_{k}^{\oplus j})$.
We prove this using the Fadell--Husseini ideal-valued index theory \cite{husseini1988}, 
for the group $(\Z/2)^k$ and $\F_2$ coefficients.

Let $(\Z/2)^k=\langle \varepsilon_1,\ldots,\varepsilon_k\rangle$ with $\varepsilon_i$ acting antipodally on the $i$-th sphere in the product $Y_{d,k}=(S^d)^k$.
The cohomology of $(\Z/2)^k$ is
$H^*((\Z/2)^k;\F_2)=\F_2[u_1,\ldots,u_k]$, 
where $\deg(u_i)=1$ and the variable $u_i$ corresponds to the generator 
$\varepsilon_i$, $1\leq i\leq k$.
Then according to \cite[Ex.\,3.3]{husseini1988}
\[
\ind_{(\Z/2)^k} (Y_{d,k};\F_2)=\big\langle u_1^{d+1},\ldots, u_k^{d+1}\big\rangle.
\]
According to \cite[Prop.\,3.7]{husseini1988} or \cite[Prop.\,3.13]{blagojevic2011} 
we have that 
\[
\ind_{(\Z/2)^k}(S(U_{k}^{\oplus j});\F_2)=\Big\langle  \Big(\prod_{(\alpha_1,\ldots,\alpha_k)\in (\Z/2)^k{\setminus}\{0\}}  (\alpha_1u_1+\cdots+\alpha_ku_k)\Big)^j \Big\rangle.
\]

Now assume that there is a $(\Z/2)^k$-equivariant map 
$Y_{d,k}\longrightarrow S(U_{k}^{\oplus j})$. 
Then a basic property of the Fadell--Husseini index \cite[Sec.\,2]{husseini1988} implies that 
\[
\ind_{(\Z/2)^k}(S(U_{k}^{\oplus j});\F_2)\subseteq \ind_{(\Z/2)^k} (Y_{d,k};\F_2), 
\]
and consequently 
\begin{equation}
\label{eq:upper_bound}
\Big(\prod_{(\alpha_1,\ldots,\alpha_k)\in (\Z/2)^k{\setminus}\{0\}}  (\alpha_1u_1+\cdots+\alpha_ku_k) \; \;\Big)^j\in\big\langle u_1^{d+1},\ldots, u_k^{d+1}\big\rangle.
\end{equation}
Let us denote 
\[
p=\prod_{(\alpha_1,\ldots,\alpha_k)\in (\Z/2)^k{\setminus}\{0\}}  (\alpha_1u_1+\cdots+\alpha_ku_k)\,\in\,\F_2[u_1,\dots,u_k].
\]
As a Dickson polynomial of maximal degree \cite[Sec.\,III.2]{adem_milgram2004} it can be presented as
\[ p=\sum_{\pi\in\Sym_k}u_{\pi(1)}^{2^{k-1}}u_{\pi(2)}^{2^{k-2}}\cdots u_{\pi(k)}^{2^0}.\] 
Therefore, 
\begin{eqnarray*}
p^j &=& \Big(\prod_{(\alpha_1,\ldots,\alpha_k)\in (\Z/2)^k{\setminus}\{0\}}  (\alpha_1u_1+\cdots+\alpha_ku_k)\;\; \Big)^j\\
&=&
\Big(\sum_{\pi\in\Sym_k}u_{\pi(1)}^{2^{k-1}}u_{\pi(2)}^{2^{k-2}}\cdots u_{\pi(k)}^{2^0}\Big)^{2^t + r}\\
      &=&\Big(\sum_{\pi\in\Sym_k}u_{\pi(1)}^{2^{k+t-1}}u_{\pi(2)}^{2^{k+t-2}}\cdots u_{\pi(k)}^{2^t}\Big)\Big(\sum_{\pi\in\Sym_k}u_{\pi(1)}^{2^{k-1}}u_{\pi(2)}^{2^{k-2}}\cdots u_{\pi(k)}^{2^0}\Big)^{r}\\
       &=&\big(u_{1}^{2^{k+t-1}}u_{2}^{2^{k+t-2}}\cdots u_{k}^{2^t}\big)\cdot\big(u_1^ru_2^{2r}\cdots u_k^{2^{k-1}r}     \big) +\mathrm{Rest}\\
		&=&u_{1}^{2^{k+t-1}+r}u_{2}^{2^{k+t-2}+2r}\cdots u_{k}^{2^t+2^{k-1}r} +\mathrm{Rest},
\end{eqnarray*}
where $\mathrm{Rest}$ does not contain the monomial $u_{1}^{2^{k+t-1}+r}u_{2}^{2^{k+t-2}+2r}\cdots u_{k}^{2^t+2^{k-1}r}$.
Thus $p^j\notin \langle u_1^{d+1},\ldots, u_k^{d+1}\rangle$, which contradicts \eqref{eq:upper_bound}.
This concludes the proof of the non\-existence of a $(\Z/2)^k$-equivariant map $Y_{d,k}\longrightarrow S(U_{k}^{\oplus j})$.
\end{proof}

\subsection{Dimension reductions via constraints}
\label{subsec:reductions}

In order to bound $\Delta(j,k)$ it is not always necessary to make use of 
advanced topological methods, as there are also reduction arguments available: Hadwiger and Ramos used the rather obvious fact that 
\begin{equation}
\label{eq:reduction_Ramos}
\Delta(j,k) \le \Delta(2j, k-1),
\end{equation} 
while Matschke in \cite{matschke2009note} proved that
\begin{equation}
\label{eq:reduction_Matschke}
\Delta(j,k) \le \Delta(j+1,k)-1. 
\end{equation}
We employ a simple combinatorial reduction argument to deduce the non-existence of equivariant maps
and, in particular, to obtain a topological analog of Matschke's result,
Proposition~\ref{prop:Matschke-reduction_analog}.
Recently, we used this approach to give elementary proofs of old and new Tverberg-type results \cite{blagojevic2014}.

For $\alpha\in (\Z/2)^k$ let $V_\alpha$ be the one-dimensional real $(\Z/2)^k$-representation for which $\beta \in (\Z/2)^k$ acts non-trivially if and only if $\sum_{i=1}^k \alpha_i\beta_i=1 \mod 2$. 
Then there is an isomorphism of $(\Z/2)^k$-representations 
$U_k\cong \bigoplus_{\alpha\in (\Z/2)^k{\setminus}\{0\}}V_{\alpha}$.
Denote by $A \subseteq (\Z/2)^k$ the subset of all $\alpha=(\alpha_1,\ldots,\alpha_k) \in (\Z/2)^k$ with exactly one $\alpha_i$ non-zero, and let $B \subseteq (\Z/2)^k$ be the subset of all 
$\alpha \in (\Z/2)^k$ with more than one $\alpha_i$ non-zero. 
The representation $U_k$ splits into 
$\bigoplus_{\alpha\in A} V_{\alpha} \oplus \bigoplus_{\alpha\in B} V_{\alpha}$.

\begin{proposition}\label{prop:Matschke-reduction_analog}
	If there is no $\Wk_k$-equivariant map $Y_{d,k} \longrightarrow S(U_k^{\oplus j})$, then there is also no $\Wk_k$-equivariant map $Y_{d-1,k} \longrightarrow S\big(U_k^{\oplus (j-1)} \oplus \bigoplus_{\alpha\in B} V_{\alpha}\big)$.
\end{proposition}

\begin{proof}
There is an $\Wk_k$-equivariant map $\Phi\colon Y_{d,k} \longrightarrow \bigoplus_{\alpha \in A} V_\alpha$ with $\Phi^{-1}(0) = Y_{d-1,k}$, where $Y_{d-1,k} \subseteq Y_{d,k}$ is naturally identified with a product of equators. 
In fact, the space $Y_{d,k}$ contains all real $(d+1) \times k$ matrices whose columns have norm one. 
Now define 
\[
\Phi\colon Y_{d,k} \longrightarrow \bigoplus_{\alpha \in A} V_\alpha, \qquad
A \longmapsto (x_{d+1,1}, \dots, x_{d+1,k})
\]
as the map that evaluates the last row of a given matrix $A \in Y_{d,k}$.

Let $f \colon Y_{d-1,k} \longrightarrow U_k^{\oplus (j-1)} \oplus \bigoplus_{\alpha\in B} V_{\alpha}$ be an arbitrary equivariant map. We need to show that $f$ has a zero. 
Extend $f$ somehow to an equivariant map $F\colon Y_{d,k} \longrightarrow  U_k^{\oplus (j-1)} \oplus \bigoplus_{\alpha\in B} V_{\alpha}$. 
The map $F \oplus \Phi\colon Y_{d,k} \longrightarrow U_k^{\oplus j}$ has a zero $x_0$, otherwise it would induce an $\Wk_k$-equivariant map $Y_{d,k} \longrightarrow S(U_k^{\oplus j})$ by retraction. Since $\Phi(x_0) = 0$, we have $x_0 \in Y_{d-1,k}$ and it is a zero of the map $f$.
\end{proof}

By induction we obtain the following criterion.

\begin{theorem}
  Suppose there is no $\Wk_k$-equivariant map $Y_{d,k} \longrightarrow S(U_k^{\oplus j})$, 
  then $\Delta(j-m,k) \le d-m$ for all $1\leq m \leq  j-1$.  
\end{theorem}

\section{The Ramos conjecture for $\Delta(2,2)$}
\label{sec:hadwiger}

The first result on the  Gr\"{u}nbaum--Hadwiger--Ramos problem for more than one hyperplane is due to Hadwiger \cite{hadwiger1966}. 
He proved the following result.

\begin{theorem}[Hadwiger \cite{hadwiger1966}]
Let $A, B \subseteq \R^3$ be two compact sets with positive Lebesgue measure and denote by $\mu_A$ and $\mu_B$ the restriction of the Lebesgue measure to the respective sets. 
Then there is an arrangement of two affine hyperplanes that equipart the measures $\mu_A$ and $\mu_B$.
\end{theorem}

We prove, using a simple degree-theoretic argument, that any two masses in~$\R^3$ 
can be equiparted by two affine hyperplanes, so $\Delta(2,2) \le 3$.
For this we use that equivariant maps have restricted homotopy types. 

\begin{lemma}[Equivariant Hopf Theorem~{\cite[Thm.\,II.4.11]{tom1987transformation}}]
\label{thm:equivariant}
Let $G$ be a finite group that acts on $S^d$ and acts freely on a closed oriented $d$-manifold $M$. 
Then for any two $G$-equivariant maps $\Phi, \Psi \colon M \longrightarrow S^d$ 
\[
\deg \Phi \equiv \deg \Psi \mod |G|.
\]
\end{lemma}

First we consider measures with continuous densities that have connected support and restrict to the configuration space of pairs of affine hyperplanes that simultaneously bisect both measures. 
In a second step we find a point in this configuration space equiparting the measures.  

\begin{lemma}
Let $\mu_1$ and $\mu_2$ be masses on $\R^3$.
The space $C \subset S^3$ of all oriented affine hyperplanes that simultaneously bisect both $\mu_1$ and $\mu_2$ admits a $\Z/2$-equivariant map $S^1\longrightarrow C$ where the action on the sphere $S^1$ is antipodal.
\end{lemma}

\begin{proof}
The sphere $S^3$ parametrizes all oriented affine hyperplanes in $\R^3$ including the ones at infinity.
Consider the following subspace of $S^3$:
\[
S = \{u \in S^3 \: : \: \mu_1(H_u^0) =\tfrac{1}{2} \}.
\]
The space $S$ is homeomorphic to a $2$-sphere that is invariant with respect to 
the antipodal action on $S^3$ (that is, with respect to change of orientation
of the hyperplane): Any normal vector in $\R^3$ determines a unique bisecting affine hyperplane for $\mu_1$.
For this we need that $\mu_1$ has connected support.

Let us define a map $\phi\colon S\longrightarrow \R$ by $u\longmapsto \mu_2(H_u^0)-\mu_2(H_u^1)$. 
The map $\phi$ is $\Z/2$-equivariant where the action on both spaces is antipodal.
Set $C=\phi^{-1}(0)=\bigcup_{i\in I}C_i$ where the $C_i$ are the path-components of $C$. 
First, we prove that there exists a $\Z/2$-invariant path-component $C_j$ of $C$.

According to the general Borsuk--Ulam--Bourgin--Yang Theorem~\cite[Sec.\,6.1]{b_l_z_2012}
\begin{equation}
\label{eq:BUBY}
\ind_{\Z/2}(C;\F_2)\cdot \ind_{\Z/2}(\R{\setminus}\{0\};\F_2)\subseteq\ind_{\Z/2}(S;\F_2).
\end{equation}
Let the cohomology of $\Z/2$ be denoted by $H^*(\Z/2;\F_2)=\F_2[t]$, where $\deg(t)=1$.
Using \cite[Prop.\,3.13]{blagojevic2011} we get
\[
\ind_{\Z/2}(\R{\setminus}\{0\};\F_2)=\ind_{\Z/2}(S^0;\F_2)=\langle t\rangle
\quad\text{and}\quad
\ind_{\Z/2}(S;\F_2)=\langle t^3\rangle.
\]
If $C$ did not have a path-component that the $\Z/2$-action maps to itself, then the path-components of $C$ would come in pairs
that the group action would exchange. 
Consequently, there exists a $\Z/2$-equivariant map $C\longrightarrow S^0$ implying that $\ind_{\Z/2}(C;\F_2)=\langle t\rangle$.
This contradicts \eqref{eq:BUBY}, and so $C$ contains a path-component
that the $\Z/2$-action maps to itself.

Let $C_j$ be a $\Z/2$-invariant path-component of $C$.
We prove that there exists a $\Z/2$-equivariant map $S^1\longrightarrow C_j$ where the action on $S^1$ is antipodal.
Connect two antipodal points in $C_j$ via an injective path and extend to $S^1$ via the $\Z/2$-symmetry.
\end{proof}

\begin{theorem}
\label{thm:hadwiger}
	$\Delta(2,2) = 3$.
\end{theorem}

\begin{proof}
Let $\mu_1$ and $\mu_2$ be masses on $\R^3$.
The subspace $C \subseteq S^3$ of oriented hyperplanes that simultaneously bisect both 
masses admits a $\Z/2$-equivariant map $i\colon S^1\longrightarrow C$, where the action on the sphere $S^1$ is antipodal. 

Consider the composition $\Phi\colon S^1\times S^1 \longrightarrow   C \times C \longrightarrow \R^2$ defined by
\[
(u,v) \longmapsto (\mu_1\big(H_{i(u)}^0 \cap H_{i(v)}^0\big)-\tfrac{1}{4}, \mu_2\big(H_{i(u)}^0 \cap H_{i(v)}^0\big)-\tfrac{1}{4}).
\]
Assume that $\mu_1$ and $\mu_2$ do not have any equipartition by two hyperplanes in~$\R^3$.
Consequently $0\notin \Phi(S^1\times S^1)$, 
since the zeros of the map $\Phi$ are pairs of hyperplanes that equipart $\mu_1$ and $\mu_2$.
Now $\Phi$ composed with radial retraction $\R^2{\setminus}\{0\}\longrightarrow S^1$ induces the map $\Psi\colon  S^1 \times S^1 \longrightarrow S^1$.
Notice that $\Psi(u,u) = \big(\frac{\sqrt{2}}{2}, \frac{\sqrt{2}}{2}\big)$ for each $u \in S^1$. 
Thus  the map $\Psi|_D \colon D \longrightarrow S^1$, where $D = \{(u,u) : u \in S^1\}$ is the diagonal, is constant and so has degree $0$.
 
Let $t$ be a generator of $\Z/4$.
Then $t \cdot (u,v) = (v,-u)$ defines a free $\Z/4$-action on $S^1 \times S^1$. 
The circle $\Gamma = \{(u,e^{i\frac{\pi}{2}}\cdot u) : u \in S^1\} \subseteq S^1 \times S^1$ is a $\Z/4$-invariant subspace that is homotopic to the diagonal $D$ in $S^1 \times S^1$. 
Thus $\deg \Psi|_\Gamma = \deg \Psi|_D = 0$. 

On the other hand, the map $\Psi|_\Gamma : \Gamma \longrightarrow S^1$ is $\Z/4$-equivariant with the generator $t$ acting antipodally on the codomain sphere $S^1$. 
All such maps have the same degree modulo $4$ by Lemma~\ref{thm:equivariant} and $z \mapsto z^2$ is such a map of degree $2$. This yields a contradiction, and so the map $\Phi$ has a zero.
\end{proof}

The reduction argument \eqref{eq:reduction_Ramos} applied to the result of the previous theorem in combination with Ramos' lower bound yields the following consequence.

\begin{corollary}[Hadwiger \cite{hadwiger1966}]
$\Delta(1,3)=3$.
\end{corollary}

\section{The Ramos conjecture for $\Delta(2^t+1,2)$}
\label{sec : D(2t+1,2)}

In this section we prove the following theorem, establishing a family of exact values for the function $\Delta(j,2)$ in the case of two hyperplanes.
It is a nontrivial 
instance of the Ramos conjecture that was previously claimed
by \v{Z}ivaljevi\'c \cite[Thm.\,2.1]{zivaljevic2011equipartitions}~\cite[Thm.\,2.1]{zivaljevic2011equipartitions-published},
but the proof given there is not complete; see Section \ref{sec:further_gaps}.

\begin{theorem}\label{thm:Delta(2^t+1,2)} 
 $\Delta(2^t + 1, 2) = 3\cdot 2^{t-1}+2$ for any $t \ge 2$.
\end{theorem}

Using the reduction of \eqref{eq:reduction_Matschke} we obtain from this that
\[
    \Delta(2^t, 2) \le 3\cdot 2^{t-1}+1\quad \textrm{for any }t \ge 2.
\]
as listed in Table \ref{table:summary}.

The rough outline of the proof of Theorem \ref{thm:Delta(2^t+1,2)} is as follows: 
For $d=3\cdot 2^{t-1}+2$ consider $j$ masses in $\R^d$ that do not admit an equipartition by two affine hyperplanes.
They induce a $D_8$-equivariant test map $\psi\colon S^d \times S^d \longrightarrow S^{2d-2}$. 
The restricted map $\tildepsi\colon S^{d-1} \times S^{d-1} \longrightarrow S^{2d-2}$ has degree zero since it factors through $S^d \times S^d$. We will then consider the test map $\phi$ for $j$ specific masses and compute the degree of the restricted map $\tildephi$ on $S^{d-1} \times S^{d-1}$ by counting the zeros of $\phi$ on $B^d \times S^{d-1}$ (where $B^d$ is a hemisphere of~$S^d$) with sign and multiplicity. This is done by counting equipartitions for this specific set of measures. 
The maps $\tildepsi$ and $\tildephi$ need not be homotopic and so their degrees might not coincide. This is remedied by exploiting the equivariance of both maps, yielding $\deg\tildepsi\equiv\deg\tildephi \mod 8$, which gives a contradiction if $j-1$ is
a power of two, $j\ge5$.

\subsection{Equipartitions restrict degrees of equivariant maps}

In order to show that $\Delta(j,k) \le d$ we use Proposition~\ref{prop:CS/TM}(2)
and prove that there is no 
$\Wk_k$-equivariant map $Y_{d,k} \longrightarrow S(U_{k}^{\oplus j})$. 

\begin{lemma}
\label{lemma : a-1}
Let $\Delta(j,k) > d$ for $k(d-1) = (2^k-1)j-1$ and assume that $k(d-1)$ is 
not divisible by $d$. 
Then there is an $\Wk_k$-equivariant map $\tildepsi\colon Y_{d-1,k} \longrightarrow S(U_{k}^{\oplus j})$ with $\deg \tildepsi = 0$. 
\end{lemma}

\begin{proof}
Since $\Delta(j,k) > d$ there is an $\Wk_k$-equivariant map $\psi\colon Y_{d,k} \longrightarrow S(U_{k}^{\oplus j})$. 
This map restricts to an $\Wk_k$-equivariant map $\tildepsi\colon  Y_{d-1,k} \longrightarrow S(U_{k}^{\oplus j})$ on the product of the equators. 
The domain and codomain of $\tildepsi$ are closed orientable manifolds of the same dimension, and thus $\tildepsi$ has a well-defined degree up to a sign.
Consider the following commutative diagram of $\Wk_k$-equivariant maps
\[
\xymatrix{
Y_{d,k}=(S^d)^k\ar[r]^{\psi}  & S(U_{k}^{\oplus j})\\
Y_{d-1,k}=(S^{d-1})^k\ar@{^{(}->}[u] \ar[ur]^{\tildepsi}.&
}
\]
After applying the $k(d-1)$-dimensional homology functor we get
\[
\xymatrix{
H_{k(d-1)}((S^d)^k;\Z)\ar[r]^{\psi_*}  & H_{k(d-1)}(S(U_{k}^{\oplus j});\Z)\\
H_{k(d-1)}((S^{d-1})^k;\Z)\ar[u] \ar[ur]^{\tildepsi_*}.&
}
\]
Thus the map $\tildepsi_\ast$ factors through $H_{k(d-1)}((S^d)^k;\Z)$.
Since $d$ does not divide $k(d-1)$ we have that $H_{k(d-1)}((S^d)^k;\Z) \cong 0$. Consequently, $\deg \tildepsi = 0$.
\end{proof}

The equality $k(d-1) = (2^k-1)j-1$ implies that $d = \frac{(2^k-1)}{k}j-\frac{1}{k}+1 = \lceil \frac{(2^k-1)}{k} j\rceil$, which coincides with the lower bound~\eqref{ramos_bound}. The space $Y_{d-1,k} = (S^{d-1})^k$ is naturally a subspace of $Y_{d,k}$ by identifying it with oriented linear hyperplanes in $\R^d$, that is, $(x_1, \dots, x_k) \in Y_{d,k} \subseteq (\R^{d+1})^k$ is in $Y_{d-1,k}$ precisely if $\langle e_{d+1}, x_i \rangle = 0$ for $i =1, \dots, k$.

We use the following generalized equivariant Hopf theorem.

\begin{theorem}[Kushkuley \& Balanov {\cite[Cor.\,2.4]{balanovLNM96}}]
\label{thm_generalized_equiv_hopf}
Let $M$ be a compact oriented $n$-dimensional manifold with an action of a finite group $G$. 
Let $N \subseteq M$ be a closed $G$-invariant subset containing the set of all points with non-trivial stabilizers. Then any two $G$-equivariant maps $\phi, \psi\colon M \longrightarrow S^n$ that are equivariantly homotopic on $N$ satisfy $\deg \phi \equiv \deg \psi \mod |G|$.
\end{theorem}

The set $Y_{d-1,k}^{>1}$ of points in $Y_{d-1,k}$ with non-trivial stabilizers with respect to the action of $\Wk_k$ is
\[
\{(x_1, \ldots, x_k)\in Y_{d-1,k} : x_r = x_s\text{ or } x_r = -x_s\text{ for some }r\neq s\}.
\]
Observe that for $k\geq 3$ and $d \ge 2$, the space $Y_{d-1,k}^{>1}$ is path-connected, while for $k=2$ it consists of two path-components.

\begin{corollary}
\label{cor:bounds-from-maps}
Let $k(d-1) = (2^k-1)j-1$ and let $k(d-1)$ be not divisible by $d$. 
Let $\mathcal M = (\mu_1, \dots, \mu_j)$ be a collection of masses on $\R^d$ that cannot be equiparted by $k$ linear hyperplanes with the corresponding test map 
$\phi = \phi_{\mathcal M}\colon Y_{d,k} \longrightarrow U_k^{\oplus j}$. 
Denote the (normalized) test map restricted to linear hyperplanes by $\tildephi\colon Y_{d-1,k} \longrightarrow S(U_k^{\oplus j})$. If $\deg\tildephi\not\equiv 0 \mod 2^kk!$, 
then $\Delta(j,k) = d$, that is, the Ramos conjecture holds for $j$ masses and $k$ hyperplanes.
\end{corollary}

\begin{proof}
Suppose $\Delta(j,k) > d$.
Then from Lemma~\ref{lemma : a-1} we get an $\Wk_k$-equivariant map $\tildepsi\colon Y_{d-1,k} \longrightarrow S(U_{k}^{\oplus j})$ with $\deg \tildepsi = 0$.
By assumption there is an $\Wk_k$-equivariant map $\tildephi\colon Y_{d-1,k}\longrightarrow S(U_{k}^{\oplus j})$ with $\deg \tildephi \not\equiv 0 \mod |\Wk_k|$.
Set $N=Y_{d-1,k}^{>1}$.
Once we have shown that $\tildephi$ and $\tildepsi$ are equivariantly homotopic on $N$ we can apply Theorem~\ref{thm_generalized_equiv_hopf} and get that $\deg \tildephi \equiv \deg \tildepsi \mod |\Wk_k|$.
This is a contradiction with $\deg \tildepsi = 0$, and therefore $\Delta(j,k) \le d$.

The equivariant homotopy from $\tildephi|_N$ to $\tildepsi|_N$ is just the linear homotopy in $U_k^{\oplus j}$ normalized to the unit sphere. 
For this to be well-defined we need to show that the linear homotopy does not have a zero. This follows from the fact that for each point $z \in N$ the vectors $\phi(z)$ and $\psi(z)$ lie in some affine subspace of $U_k^{\oplus j}$ that is not a linear subspace. Since $z = (x_1, \dots, x_k) \in N$ has non-trivial stabilizer there are $r \neq s$ with $x_r = \pm x_s$. Thus the corresponding affine hyperplanes $H_r$ and $H_s$ coincide with perhaps opposite orientations. 

Recall that $U_k$ can be written as $\bigoplus_\alpha V_\alpha$, where the direct sum is taken over all $\alpha\in(\Z/2)^k$ with at least one $1$, and $V_\alpha$ is the one-dimensional real $(\Z/2)^k$-module, where $\beta \in (\Z/2)^k$ acts nontrivially precisely if $\sum \alpha_s\beta_s = 1$. 
Let $\alpha\in (\Z/2)^k$ be the element with $\alpha_s = 1 = \alpha_r$ and $\alpha_\ell = 0$ for all other indices $\ell$. 
Then for $x_s = x_r$ both maps $\phi$ and $\psi$ map the point $z$ to $1$ in the summand $V_\alpha \subseteq U_k$. For $x_s = -x_r$ the values $\psi(z)$ and $\phi(z)$ are $-1$ in the $V_\alpha$-components.
\end{proof}

\subsection{The standard configuration along the moment curve}

Now we specialize to the problem of two hyperplanes, $k = 2$. 
In this case the relevant group is the dihedral group $\Wk_2=D_8=(\Z/2)^2\rtimes \Z/2=\langle \varepsilon_1,\varepsilon_2\rangle \rtimes \langle\omega\rangle$, and the corresponding test space is $Y_{d,2}=S^d\times S^d$.
Thus the test map is a $D_8$-equivariant map $\phi\colon S^d \times S^d \longrightarrow U_2^{\oplus j}$ whose zeros correspond to equipartitions. 

Before proceeding further we recall how, in this case, $D_8=(\Z/2)^2\rtimes \Z/2=\langle \varepsilon_1,\varepsilon_2\rangle \rtimes \langle\omega\rangle$ acts on $S^d \times S^d$ and $U_2$.
For $(u,v)\in S^d \times S^d$ we have that
\[
\varepsilon_1\cdot(u,v)=(-u,v),\quad \varepsilon_2\cdot(u,v)=(u,-v),\quad \omega\cdot(u,v)=(v,u).
\]
The real $3$-dimensional $D_8$-representation $U_2$ considered as a $(\Z/2)^2$-representation decomposes into a direct sum of irreducible real $1$-dimensional representations as 
$U_2=V_{(1,0)}\oplus V_{(0,1)}\oplus V_{(1,1)}$,
where $V_{(1,0)}=V_{(0,1)}=V_{(1,1)}=\R$ and 
\[
\varepsilon_1\cdot(a,b,c)=(-a,b,-c),\quad \varepsilon_2\cdot(a,b,c)=(a,-b,-c),\quad \omega\cdot(a,b,c)=(b,a,c)
\]
for $(a,b,c)\in V_{(1,0)}\oplus V_{(0,1)}\oplus V_{(1,1)}$.

We will now define masses $\mu_1, \dots, \mu_j$ for which computing the degree of the normalized test map restricted to linear hyperplanes is particularly simple.
Recall that the \emph{moment curve} $\gamma(t) = (t, \ldots, t^d)$ in $\R^d$ has the special property that any set of pairwise distinct points on $\gamma$ is in general position. 
Hence, every affine hyperplane intersects $\gamma$ in at most $d$ points. 
For the rest of this section we consider the masses $\mu_1, \ldots, \mu_j$ to be concentrated along $j$ pairwise disjoint intervals along the moment curve
that do not include the origin.

The masses $\mu_1, \dots, \mu_j$ satisfy the hypotheses of Corollary \ref{cor:bounds-from-maps} for $k = 2$ and $2d = 3j+1$: 
Any equipartition of $\mu_1, \ldots, \mu_j$ by two affine hyperplanes intersects the moment curve in $3j$ points. Additionally requiring that both hyperplanes pass through the origin prescribes one more intersection point with $\gamma$ for each hyperplane. 
Two hyperplanes intersect the moment curve in at most $2d$ points, that is, the space of linear hyperplanes $Y_{d-1,2}$ contains no pair of equiparting hyperplanes if $2d < 3j+2$. 
Now we will compute the degree of the restricted test map by counting equipartitions.

\begin{lemma}\label{lemma:number-of-equipartitons}
Let $2d = 3j+1$ and $\mu_1, \ldots, \mu_j$ be masses concentrated on the
pairwise disjoint intervals $\gamma([2,3]),\ldots,\gamma([2j,2j+1])$ of length $1$ along the moment curve in $\R^d$. 
Then there are $\binom{j}{\frac{j-1}{2}}$ pairs of unoriented (non-parallel) affine hyperplanes $(H_1, H_2)$ equiparting $\mu_1, \dots, \mu_j$ such that $H_2$ passes through the origin.
\end{lemma}

\begin{proof}
To equipart $\mu_1, \ldots, \mu_j$ the pair $(H_1, H_2)$ needs to have at least $3j$ intersection points with the moment curve. Moreover, $H_2$ is a linear hyperplane. 
Thus hyperplanes $(H_1, H_2)$ intersect the moment curve in at least $3j+1$ points. 
Since $2d = 3j+1$ and every hyperplane can intersect in at most $d$ points, there are exactly $3j+1$ intersection points. 
In particular, each $\mu_i$ has either one intersection with $H_1$ (in the midpoint of $\mu_i$) and two intersections with $H_2$ (in the midpoint of the two halves defined by $H_1$) or vice versa.
Consequently, the intersection points of the pair $(H_1, H_2)$ with the interval $\mu_i$ are uniquely determined by the number of intersections of $\mu_i$ and $H_1$. 
There are $\binom{j}{2j-d}$ masses with exactly one point of intersection with $H_1$. 
Since $d = \frac{3j+1}{2}$ this is equal to $\binom{j}{\tfrac{j-1}{2}}$.
\end{proof}

\subsection{Computing the degree of the restricted test map geometrically}

Let $\phi = \phi_{\mathcal M}\colon S^d \times S^d \longrightarrow U_2^{\oplus j}$ be the $D_8$-equivariant test map associated to the standard configuration $\mathcal M$ of $j$ masses along the moment curve in $\R^d$ where $2d = 3j+1$. By Lemma~\ref{lemma:number-of-equipartitons} such an equipartition exists and thus $\phi^{-1}(0)$ is non-empty. However there is no such equipartition by linear hyperplanes since this would require more than $d$ intersection points of some hyperplane with the moment curve $\gamma$. 

Denote by $\tildephi\colon S^{d-1} \times S^{d-1} \longrightarrow S(U_2^{\oplus j})$ 
the normalized restriction of $\phi$ to linear hyperplanes. 
Note that $\dim S^{d-1} \times S^{d-1} = 2d-2 = 3j-1 = \dim S(U_2^{\oplus j})$ and
thus $\tildephi$ has well-defined degree (up to a sign).
For even $d$ this degree modulo $8$ was previously computed by 
\v{Z}ivaljevi\'c \cite[Prop.\,9.15]{zivaljevic2011equipartitions}.

\begin{lemma}\label{lem:degrees-test-maps}
	For even $d$ the map $\tildephi\colon S^{d-1} \times S^{d-1} \longrightarrow S(U_2^{\oplus j})$ has degree 
	\[\deg \tildephi = 2 \binom{j}{\tfrac{j-1}{2}}.\]
	For odd $d$ the degree of $\tildephi$ vanishes.
\end{lemma}

We will now prove this lemma by counting zeros of $\phi$ with signs and multiplicities. Theorem \ref{thm:Delta(2^t+1,2)} then follows from an application of Corollary \ref{cor:bounds-from-maps} once we have established that $2 \binom{2^t+1}{2^{t-1}}$ is not divisible by $8$ for $t \ge 2$.

\begin{proof}[Proof of Lemma~\ref{lem:degrees-test-maps}]
Let $W \subseteq S^d \times S^d$ be the subspace of hyperplanes $(H_1,H_2)$, where $H_1$ has the origin in its positive half-space and $H_2$ is a linear hyperplane. The subspace $W$ is a manifold homeomorphic to $B^d \times S^{d-1}$ with boundary $S^{d-1} \times S^{d-1}$. By Lemma~\ref{lemma:number-of-equipartitons} $\phi$ has $2  \binom{j}{\tfrac{j-1}{2}}$ zeros on $W$: The orientation of $H_1$ is prescribed by the requirement that the origin be in its positive half-space, but the orientation of $H_2$ is not prescribed. 
We will show that for $d$ even all local degrees of $\phi$ on $W$ are $1$ and that $\deg \tildephi$ is the sum of local degrees of $\phi$ on $W$.

Denote by $\widetilde{W} = W \setminus \phi^{-1}(B_\epsilon(0))$ for a sufficiently small $\epsilon > 0$ such that $W \setminus \phi^{-1}(0)$ deformation retracts to $\widetilde{W}$. 
The boundary $\partial \widetilde{W}$ consists of $Y_{d-1,2}$ and disjoint copies of $(2d-2)$-spheres $S_1, \dots , S_{\ell}$, one for each zero of $\phi$ on $W$.
Let $\phi'\colon\widetilde{W}\longrightarrow S(U_2^{\oplus j})$ denote the composition of $\phi$ and radial retraction restricted to $\widetilde{W}$.
The fundamental class $[Y_{d-1,2}]$ is equal to $\sum [S_i]$ in $H_{2d-2}(\widetilde{W})$ since $Y_{d-1,2}$ and $\bigcup S_i$ are cobordant in $\widetilde{W}$. 
Now $\sum \phi'_\ast([S_i]) = \phi'_\ast([Y_{d-1,2}]) = \deg\tildephi\cdot[S(U_2^{\oplus j})]$, and hence $\deg\tildephi = \sum \deg \phi'|_{S_i}$, consult \cite[Prop.\,IV.4.5]{outerelo2009}.

That local degrees of $\phi$ are $\pm 1$ is simple to see since in a small neighborhood $U$ around any zero $(u,v)$ the test map $\phi$ is a continuous bijection: 
For any sufficiently small vector $w \in \R^{3j}$ there is exactly one tuple $(u',v') \in U$ with $\phi(u',v') = w$. 
Thus $\phi |_{\partial U}$ is a continuous bijection into some $(3j-1)$-sphere around the origin and by compactness of $\partial U$ is a homeomorphism.

The symmetry of the configuration allows us to compute the local signs of the test map. 
First let us describe a neighborhood of every zero of the test map in $W$. 
Let $(u,v) \in W$ with $\phi(u,v) = 0$. 
Denote the intersections of $H_u$ with the moment curve by $x_1, \ldots, x_d$ in the correct order along the moment curve. 
Similarly, let $y_1, \ldots, y_d$ be the intersections of $H_v$ with the moment curve. 
In particular, $y_1 = 0$. 
Choose an $\epsilon > 0$ such that $\epsilon$-balls around the $x_1, \ldots, x_d$ and around $y_2, \dots, y_d$ are pairwise disjoint and such that these balls intersect the moment curve only in precisely one interval $\mu_i$.

Tuples of hyperplanes $(H_{u'}, H_{v'})$ with $(u',v') \in W$ that still intersect the moment curve in the corresponding $\epsilon$-balls parametrize a neighborhood of $(u,v)$. 
The local neighborhood consisting of pairs of hyperplanes with the same orientation still intersecting the moment curve in the corresponding $\epsilon$-balls can be naturally parametrized by $\prod_{i=2}^{2d}  (-\epsilon, \epsilon)$, where the first $d$ factors correspond to neighborhoods of the $x_i$ and the last $d-1$ factors to $\epsilon$-balls around $y_2, \dots, y_d$. A natural basis of the tangent space at $(u,v)$ is obtained via the push-forward of the canonical basis of $\R^{2d-1}$ as tangent space at the origin. 

Consider the subspace $Z \subseteq W$ that consists of pairs of hyperplanes $(H_u, H_v)$ 
in~$W$ that each intersect the moment curve in $d$ points. 
It has two path-components determined by the orientation of $H_v$. 
The path-components of $Z$ are contractible as each hyperplane can be continuously moved to intersect the moment curve in $d$ fixed points. 
On each part the orientation around the zeros given above derives from the same global orientation since the given bases of tangent spaces transform into one another along this contraction path. 
The map $\varepsilon_2\colon (H_u, H_v) \longmapsto (H_u, H_{-v})$ is orientation-preserving if and only if $d$ is even.

Any two neighborhoods of distinct zeroes of the test map $\phi$ can be mapped onto each other by a composition of coordinate charts since their domains coincide. This is a smooth map of degree $1$: the Jacobian at the zero is the identity map. Let $(u,v)$ and $(x,y)$ be zeroes in the same path-component of $Z$ of the test map $\phi$ and let $\Psi$ be the change of coordinate chart described above. Then $\phi$ and $\phi\circ\Psi$ differ in a neighborhood of $(u,v)$ just by a permutation of coordinates. 
This permutation is always even by the following: 

\begin{claim*}
{Let $A$ and $B$ be finite sets of the same cardinality. 
Then the cardinality of the symmetric sum $A \vartriangle B$ is even.}
\end{claim*}

Up to orientation of $H_u$ the hyperplanes $H_u$ and $H_v$ are completely determined by the set of measures that $H_u$ cuts once. 
Let $A \subseteq \{1, \dots, j\}$ be the set of indices of measures that $H_u$ intersects once, and let $B \subseteq \{1, \dots, j\}$ be the same set for $H_v$. 
Then $\Psi$ is a composition of a multiple of $A \vartriangle B$ transpositions and, hence, an even permutation.

The linear map $\varepsilon_2\colon U_2^{\oplus j}\longrightarrow U_2^{\oplus j}$ always has determinant equal to $1$ since $\varepsilon_2$ is a composition of $2j$ reflections in hyperplanes on $U_2^{\oplus j}$. 
Thus for $d$ even all local degrees of $\phi$ on $W$ are the same since the coordinate change $\Psi$ preserves orientation (on a path-component),
and we have proved Lemma \ref{lem:degrees-test-maps}. Thus for $d$ even $\deg \tildephi = 2  \binom{j}{\tfrac{j-1}{2}}$.
\end{proof}

To apply Corollary \ref{cor:bounds-from-maps} it is essential to know when the binomial coefficient $\binom{j}{\tfrac{j-1}{2}}$ is divisible by $4$. 
This is answered by the following lemma by Kummer.

\begin{lemma}[Kummer \cite{kummer1852}]
\label{lemma:kummer}
Let $n \ge m \ge 0$ be integers and let $p$ be a prime. The maximal integer $k$ such that $p^k$ divides $\binom{n}{m}$ is the number of carries when $m$ and $n-m$ are added in base $p$.
\end{lemma}

Putting these statements together we obtain a proof of Theorem \ref{thm:Delta(2^t+1,2)}.

\begin{proof}[Proof of Theorem \ref{thm:Delta(2^t+1,2)}]
Let $k=2$, $j= 2^t +1$ with $t\ge2$, and $d = 3\cdot 2^{t-1}+2$. 
Then $2(d-1) = j(2^k-1)-1$ and $d$ does not divide $k$. Thus we can apply Corollary~\ref{cor:bounds-from-maps} to the standard configuration $\mathcal M$ of $j$ masses along the moment curve. The restriction to linear hyperplanes $\tildephi$ of the corresponding test map $\phi_{\mathcal M}$ has degree $\binom{j}{\tfrac{j-1}{2}}$ by 
Lemma~\ref{lem:degrees-test-maps} since $d$ is even. This degree is non-zero modulo $8$ by Lemma~\ref{lemma:kummer}.
\end{proof}

\section{The failure of the free configuration space}
\label{sec : the failure}

Here we prove the following theorem about the existence of $\Wk_k$-equivariant maps from the free configuration space $Z_{d,k}$.
Recall that $Z_{d,k}=\{(x_1,\dots,x_k)\in Y_{d,k}: x_s\neq\pm x_r\text{\,for\,}s<r\}$ 
is the largest subspace of $Y_{d,k}$ on which the $\Wk_k$-action is free.

\begin{theorem}
\label{theorem_existence_W_k_maps}
Let $d\geq k\geq 3$ and $(2^k-1)j+2\geq \max\{dk,dk+4-k\}$. 
Then there is an $\Wk_k$-equivariant map $Z_{d,k}\longrightarrow S(U_{k}^{\oplus j})$.
\end{theorem}

Theorem~\ref{theorem_existence_W_k_maps} will be proved in Section~\ref{subsec:Proof_of_Thm6.1}.
As $\dim S(U_{k}^{\oplus j})=(2^k-1)j-1$ and $\dim Z_{d,k}=dk$, it exhibits a disadvantage of the free configuration spaces.

As a direct consequence of Theorem~\ref{theorem_existence_W_k_maps} we 
prove the first main result
claimed in \v{Z}ivaljevi\'c's 2008 paper, \cite[Thm.\,5.9]{zivaljevic2008}. 

\begin{corollary}
\label{cor:Rade-01}
There is an $\Wk_k$-equivariant map $f\colon Z_{4,4}\longrightarrow S(U_4)$.
\end{corollary}

\noindent 
In Section~\ref{sec:gap-lemma-4.3} we explain why the proof 
given in \cite{zivaljevic2008} for this result is invalid.
For comparison, $Z_{4,4}$ is there denoted by $(S^4)^4_{\delta}$.
Furthermore, in Section~\ref{sec:gap-th-5.1} we exhibit a gap in the proof of the second main (positive) result of the same paper, \cite[Thm.\,5.1]{zivaljevic2008}.

\subsection{Existence of equivariant maps}
Let $G$ be a finite group, let $X$ be a free $G$-CW complex and $W$ be an orthogonal real $G$-representation.
Let us further denote by $\mathrm{cohdim}\,X=\max\{i:H^i(X;\Z)\neq 0\}$ the \emph{cohomological dimension} of the space~$X$.

In this section we consider the existence of a $G$-equivariant map $X \longrightarrow S(W)$ under specific conditions and prove the following theorem.
\begin{theorem}
\label{theorem:existence_of_G-map}
Let $G$ be a finite group, let $X$ be a free $G$-CW complex, let $W$ be an orthogonal real $G$-representation, and let $I = \{i:\dim W-1\leq i\leq\dim X-1\}$.
If 
\begin{compactenum}[\rm (i)]
\item $2 \le \mathrm{cohdim}\,X < \dim W$, and
\item $\pi_i S(W)$ is a trivial $\Z[G]$-module for every $i \in I$,
\end{compactenum}
then there exists a $G$-equivariant map $X \longrightarrow S(W)$.
\end{theorem}

The proof of the theorem will be obtained via equivariant obstruction theory, as presented by tom Dieck in \cite[Sec.\,II.3]{tom1987transformation}. 
In the proof of the theorem we use the following special case of a result given as an exercise by Bredon \cite[Exer.\,9,\,p.\,168]{bredon1972introduction}. 
It is an extension (for acyclicity above a certain dimension) of the important result from Smith theory that the quotient of a compact, acyclic space by a finite group action is still acyclic.

\begin{lemma}\label{lemma:existence_of_G-map}
  Let $G$ be a finite group acting cellularly on the compact $G$-CW-complex $X$.
  If $H^i(X;\Z)=0$ for all $i > n$, then $H^i(X/G; \Z) = 0$ for all $i > n$.
\end{lemma}
 
\begin{proof}[Proof of Theorem \ref{theorem:existence_of_G-map}]
Let us denote by $N=\dim X$, $n=\mathrm{cohdim}X$ and by $w=\dim W$.
For $i\in\{0,\ldots,N\}$, the $i$-th skeleton of $X$ is denoted as usual by $X^{(i)}$.

Since $S(W)$ is $(w-2)$-connected, $(w-1)$-simple and $X$ is a free $G$-CW complex there is no obstruction for the existence of a $G$-equivariant map $f \colon X^{(w-1)} \longrightarrow S(X)$. The proof continues by induction.

The first obstruction for the extension of the map $f$ to the $w$-skeleton $X^{(w)}$ lives in the specially defined Bredon type equivariant cohomology \cite[pp.\,111--114]{tom1987transformation}:
\[
 \mathcal{H}^w_G(X;\pi_{w-1}S(W))\cong \mathcal{H}^w_G(X;\Z),
\]
Now $\pi_{w-1}S(W)\cong\Z$ is a trivial $\Z[G]$-module by the assumption of the theorem.
The isomorphism of \cite[II,\,Prop.\,9.7,\,\rm{(ii)}]{tom1987transformation} implies $\mathcal{H}^w_G(X;\Z)\cong H^w(X/G;\Z)$, where on the right we have singular cohomology. 
Since $n=\mathrm{cohdim}X < w$, by the assumption of the theorem, an application of Lemma~\ref{lemma:existence_of_G-map} gives $\mathcal{H}^w_G(X;\Z)= 0$. 
Thus $\mathcal{H}^w_G(X;\pi_{w-1}S(W))=0$, and the map $f$ can be $G$-equivariantly extended to the $w$-skeleton of $X$.

The process continues in the same way until we reach the $N$-th skeleton of $X$ since all the ambient groups $\mathcal{H}^i_G(X;\pi_{i-1}S(W))$, $i\in\{w,\ldots N\}$, for the obstructions vanish. 
\end{proof}

\subsection{Proof of Theorem~\ref{theorem_existence_W_k_maps}}\label{subsec:Proof_of_Thm6.1}
Let $d\geq k\geq 3$ and
\[
(2^k-1)j+2\geq \max\{dk,dk+4-k\}.
\]
We prove the existence of an $\Wk_k$-equivariant map 
$Z_{d,k}\longrightarrow S(U_{k}^{\oplus j})$ by direct application 
of Theorem~\ref{theorem:existence_of_G-map}.

Let $X$ be a $dk$-dimensional $\Wk_k$-CW complex with the property 
that $X\subseteq Z_{d,k}$ is an equivariant deformation retract of $Z_{d,k}$.  
Then $X$ is a $dk$-dimensional free $\Wk_k$-CW complex and it suffices to prove that there exists an $\Wk_k$-equivariant map $X\longrightarrow S(U_{k}^{\oplus j})$.

If $dk=\dim X\leq \dim S(U_{k}^{\oplus j}) =(2^k-1)j-1$ then an $\Wk_k$-equivariant map $X\longrightarrow S(U_{k}^{\oplus j})$ exists since $X$ is a free $\Wk_k$-CW complex and all obstructions vanish. 
Thus we can in addition assume that $dk-1\geq (2^k-1)j-1$.
Now 
\[
I = \{i:\dim W-1\leq i\leq\dim X-1\}=\{i: (2^k-1)j-1 \leq i\leq dk-1\}.
\]
Since $(2^k-1)j+2\geq\max\{dk,dk+4-k\}$ and $dk-1\geq (2^k-1)j-1$ we have that $|I|\leq 3$, i.e.,
\[
I\subseteq \{(2^k-1)j-1, (2^k-1)j, (2^k-1)j+1\}.
\]

The following fact is known. For completeness we give a brief proof.

\begin{claim*}
$\mathrm{cohdim}\,Z_{d,k}=(d-1)k+1$ for $d\geq k\geq 3$.
\end{claim*}

\begin{proof}
The free configuration space $Z_{d,k}$ is defined as a difference $Y_{d,k}{\setminus}Y_{d,k}^{>1}$ of an oriented $dk$-manifold $Y_{d,k}=(S^d)^k$ and the regular CW-complex $Y_{d,k}^{>1}$.

The CW-complex $Y_{d,k}^{>1}$ can be covered by a family 
\[
\mathcal{L}=\{L_{s,r}^{+}:1\leq s<r\leq k\}\cup\{L_{s,r}^{-}:1\leq s<r\leq k\}
\]
of subcomplexes 
\[
Y_{d,k}^{>1}=\bigcup_{1\leq s<r\leq k}
\big(L_{s,r}^{+}\cup L_{s,r}^{-}\big),
\]
where for $1\leq s<r\leq k$ we set
\[
L_{s,r}^{+}=\{(x_1,\ldots,x_k)\in Y_{d,k}: x_s=x_r\},
\quad
L_{s,r}^{-}=\{(x_1,\ldots,x_k)\in Y_{d,k}: x_s=-x_r\}.
\]
Every subcomplex  $L_{s,r}^{\pm}$ as well as any finite non-empty intersection of them is $(d-1)$-connected.
Therefore, by a version of the nerve lemma \cite[Th.\,6]{bjorner2003}, we have that
$\pi_r(Y_{d,k}^{>1})\cong \pi_r(\Delta(P_{\mathcal{L}}))$ for all $r\leq d-1$, where $\Delta(P_{\mathcal{L}})$ denotes the order complex of the intersection poset $P_{\mathcal{L}}$ of the family $\mathcal{L}$.
The intersection poset $P_{\mathcal{L}}$ can be identified as a subposet of the type $B$ partition lattice $\Pi_k^B$, consult Wachs \cite[Ex.\,5.3.6]{Wachs2007}. 
Moreover, $\Pi_k^B$ is a geometric semilattice, which implies that $\Delta(P_{\mathcal{L}})\simeq\bigvee S^{k-2}$.
Thus $Y_{d,k}^{>1}$ is $(k-3)$-connected.

The Poincar\'e--Lefschetz duality \cite[Cor.\,VI.8,4]{bredon2010} relates the homology of $Z_{d,k}$ with the cohomology of the pair $(Y_{d,k},Y_{d,k}^{>1})$:
\[
H_{dk-i}(Z_{d,k};\Z)\cong H^i(Y_{d,k},Y_{d,k}^{>1};\Z).
\]
Using the long exact sequence in cohomology for the pair $(Y_{d,k},Y_{d,k}^{>1})$ and the facts that $Y_{d,k}$ is $(d-1)$-connected and $Y_{d,k}^{>1}$ is $(k-3)$-connected we get that $\tilde{H}^i(Y_{d,k},Y_{d,k}^{>1};\Z)=0$ for $i\leq k-2$ and $H^{k-1}(Y_{d,k},Y_{d,k}^{>1};\Z)\cong H^{k-2}(Y_{d,k}^{>1};\Z)\neq 0$ is free abelian.
Consequently, using the universal coefficient theorem \cite[Cor.\,V.7.2]{bredon2010}, we conclude that $\mathrm{cohdim}\,Z_{d,k}=(d-1)k+1$.
\end{proof}

In order to apply Theorem~\ref{theorem:existence_of_G-map} and complete the proof we need to verify the conditions {\rm (i)} and {\rm (ii)}.
\begin{compactenum}[\rm (i)]
\item By assumption $(2^k-1)j+2\geq \max\{dk,dk+4-k\}$ and $k\geq 3$.
Consequently, 
\[
\quad\dim U_k^{\oplus j} = (2^k-1)j>dk-1 = \dim X-1 > (d-1)k+1= \mathrm{cohdim}\,Z_{d,k}.
\]
\item Since $S(U_k^{\oplus j})\approx S^{(2^k-1)j-1}$ and $I\subseteq \{(2^k-1)j-1, (2^k-1)j, (2^k-1)j+1\}$ we consider
\[
\quad \
\pi_{(2^k-1)j-1}S(U_k^{\oplus j})\cong\Z,\ \
\pi_{(2^k-1)j}S(U_k^{\oplus j})\cong\Z/2,\ \
\pi_{(2^k-1)j+1}S(U_k^{\oplus j})\cong\Z/2
\]
as $\Z[\Wk_k]$-modules. 
Since the second two groups are $\Z/2$ and therefore trivial $\Z[\Wk_k]$-modules it remains to be shown that $\Wk_k$ acts orientation preserving on $S(U_k^{\oplus j})$.\newline
Each of the generators  $\varepsilon_i$ of $(\Z/2)^k$ acts on the top integral homology of the sphere $S(U_k^{\oplus j})$ by multiplication with 
\[
(-1)^{j\big(\binom{k-1}{0}+\binom{k-1}{1}+\cdots+\binom{k-1}{k-1}\big)}=1.
\]
Furthermore, each of the transpositions $\tau_{sr}=(sr)$ for $1\leq s<r\leq k$, 
which generate $\Sym_k$, acts on the top integral homology of the sphere $S(U_k^{\oplus j})$ by multiplication with
\[
(-1)^{j\big(\binom{k-2}{0}+\binom{k-2}{1}+\cdots+\binom{k-2}{k-2}\big)}=1.
\]
Thus $\Wk_k$ preserves orientation of $S(U_k^{\oplus j})$ and consequently $\pi_{(2^k-1)j-1}S(U_k^{\oplus j})$ is a trivial $\Z[\Wk_k]$-module.
\end{compactenum}
Now Theorem~\ref{theorem:existence_of_G-map} implies the existence of 
an $\Wk_k$-equivariant map $Z_{d,k}\longrightarrow S(U_{k}^{\oplus j})$, and we 
have completed the proof of Theorem~\ref{theorem_existence_W_k_maps}.

\subsection{Gaps in~\cite{zivaljevic2008}}
In this section we exhibit and explain essential gaps in \cite{zivaljevic2008} that invalidate 
 \v{Z}ivaljevi\'c's proofs for both main results of that paper.

\subsubsection{A Gap in \cite[Lem.\,4.3]{zivaljevic2008}}
\label{sec:gap-lemma-4.3}
We note that this lemma is the starting point
for the explicit calculations related to both main results of that paper
and thus crucial for their validity.

\noindent
First we recall some notation from \cite{zivaljevic2008}:
\begin{compactitem}[~~$\bullet$]
\item $(S^n)^n_{\delta}=\{x\in (S^n)^n : x_i\neq \pm x_j\text{ for }i\neq j\}$, consult \cite[(2.2)]{zivaljevic2008}; in notation of the current paper this is $(S^n)^n_{\delta}=Z_{n,n}$.
\item $(S^4)^4_{\delta}/\Wk_4 =: SP^4_{\delta}(\RP^4)\subseteq SP^4(\RP^ 4)$ where $SP^m(X)=X^m/\Sym_m$ denotes the symmetric product of $X$, consult \cite[Prop.\,3.1]{zivaljevic2008}.
\end{compactitem}
The following statement is claimed to be ``an easy consequence of Poincar\'e duality''; the homology is considered with coefficients in the field $\Z/2$.
\begin{quote}
{\small \cite[{\bf Lem.\,4.3}]{zivaljevic2008}\newline \emph{There is an isomorphism $H_2(SP^4_{\delta}(\RP^4))\longrightarrow H_2(SP^4(\RP^4))$ of homology groups, induced by the inclusion map $SP^4_{\delta}(\RP^4)\lhook\joinrel\longrightarrow SP^4(\RP^4)$.}}
\end{quote}
Further on, it was claimed that
\begin{multline*}
H_2(SP^4(\RP^4))\cong H_2(SP^4(\RP^{\infty}))\cong \\
H_2(K(\Z/2,1)\times K(\Z/2,2)\times K(\Z/2,3)\times K(\Z/2,4))\cong \Z/2\oplus\Z/2.
\end{multline*}
Now we prove that $H_2(SP^4_{\delta}(\RP^4))$ is not isomorphic to $\Z/2\oplus\Z/2$.
Indeed, there is a sequence of isomorphisms 
\[\begin{array}{r@{\ }c@{\ }l@{\ }l}
H_2(SP^4_{\delta}(\RP^4))  
&\cong & H^2(SP^4_{\delta}(\RP^4))&\text{\quad by the Universal Coefficient Theorem,}\\
&\cong & H^2((S^4)^4_{\delta}/\Wk_4)&\text{\quad by definition of }SP^4_{\delta}(\RP^4),\\
&\cong & H^2(\E \Wk_4\times_{\Wk_4}(S^4)^4_{\delta})&\text{\quad since the action of }\Wk_4\text{ is free,}\\
&\cong & H^2(\Wk_4)&\text{\quad since }(S^4)^4_{\delta}\text{ is }2\text{-connected \cite{feichtner}.}
\end{array}
\]
A result of Nakaoka \cite[Thm.\,5.3.1]{evens1991} combined with   
$H^2(\Sym_4)\cong \Z/2\oplus\Z/2$ \cite[Ex.\,VI.1.13]{adem_milgram2004} implies that
\begin{eqnarray*}
H^2(\Wk_4) 
&\cong & \bigoplus_{p=0}^2H^p(\Sym_4,H^{2-p}((\Z/2)^4))\\
&\cong & H^0(\Sym_4,H^2((\Z/2)^4))\oplus H^1(\Sym_4,H^1((\Z/2)^4))\oplus H^2(\Sym_4,H^0((\Z/2)^4))\\
&\cong & H^2((\Z/2)^4)^{\Sym_4} \oplus H^1(\Sym_4,H^1((\Z/2)^4))\oplus  H^2(\Sym_4)\\
&\cong & \Z/2\oplus\Z/2 \oplus H^1(\Sym_4,H^1((\Z/2)^4))\oplus  \Z/2\oplus\Z/2.
\end{eqnarray*}
Thus $H_2(SP^4_{\delta}(\RP^4))$ is not isomorphic to $\Z/2\oplus\Z/2$ and therefore \cite[Lem.\,4.3]{zivaljevic2008} is not true.

\subsubsection{A Gap in the proof of \cite[Thm.\,5.1]{zivaljevic2008}}
\label{sec:gap-th-5.1}
Here we discuss a gap in the proof of the following theorem, the second main result in \cite{zivaljevic2008}.  

\begin{quote}
{\small
\cite[{\bf Theorem 5.1}]{zivaljevic2008}\newline \emph{Suppose that $\mu$ is a measure on $\R^4$ admitting a $2$-dimensional plane of symmetry in the sense that for some $2$-plane $L\subset\R^4$ and the associated reflection $R_L\colon\R^4\longrightarrow\R^4$, for each measurable set $A\subset\R^4$, $\mu(A)=\mu(R_L(A))$.
Then $\mu$ admits a $4$-equipartition.}}
\end{quote}

\noindent
The proof of the theorem is based on \cite[Claim on p.~165]{zivaljevic2008}.
For convenience we copy the claim with the first two sentences of its proof 
from~\cite{zivaljevic2008}.

\begin{quote}
{\small
\cite[{\bf Claim on p.~165}]{zivaljevic2008} \newline \emph{There does not exist a $G$-equivariant map $f\colon (S^4)^4_{\Delta}\longrightarrow S(U_4\oplus\lambda)$, where $S(U_4\oplus\lambda)$ is the $G$-invariant unit sphere in $U_4\oplus\lambda$. 
In other words each $G$-invariant map $f\colon (S^4)^4_{\Delta}\longrightarrow U_4\oplus\lambda$ has a zero.}\\[1mm]
\emph{Proof of the Claim.} The claim is equivalent to the statement that the vector bundle $\xi\colon (S^4)^4_{\Delta}\times_G (U_4\oplus\lambda)\longrightarrow (S^4)^4_{\Delta}/G$ does not admit a non-zero continuous cross section.
For this it is sufficient to show that the top Stiefel--Whitney class $w_n(\xi)$ is non-zero.
}
\end{quote}

\noindent
The group $G$ is the direct sum $\Wk_4\oplus\Z/2$, and $(S^4)^4_{\Delta}$ is the largest subspace of $(S^4)^4$ on which the group $G$ acts freely.
The base space $(S^4)^4_{\Delta}/G$ of the vector bundle $\xi$ is an open manifold of dimension~$16$.
The real $G$-representation  $U_4\oplus\lambda$ is $16$-dimensional and therefore $\xi$ is a $16$-dimensional vector bundle.
Thus the top Stiefel--Whitney class $w_{16}(\xi)$ lives in $H^{16}((S^4)^4_{\Delta}/G;\Z/2)=0$ and so it vanishes. 
This contradicts the proof of the claim.

\smallskip
\noindent
Actually, more is true: Since $\xi$ is a $16$-dimensional vector bundle over a connected non-compact $16$-dimensional manifold $(S^4)^4_{\Delta}/G$, an exercise from
Koschorke~\cite[Exer.\,3.11]{koschorke1981vector} guarantees the existence of a nowhere vanishing cross section, again contradicting the proof of the claim.

\section{A gap in~\cite{ramos1996equipartition}}
\label{sec:ramos}

In this section we will give a counterexample to \cite[Lem.\,6.2]{ramos1996equipartition}, from which Ramos derives his main result \cite[Thm.\,6.3]{ramos1996equipartition} by induction. Our Counterexample~\ref{ex_counterexample_lemma_6_2} exploits the fact that a certain coordinate permutation action has fixed points, a crucial fact that is missed in the proof of \cite[Lem.\,6.2]{ramos1996equipartition}. 

The following table lists bounds for $\Delta(j,k)$ that are obtained directly from \cite[Thm.\,6.3]{ramos1996equipartition}. 
They cannot be obtained from \cite[Thm.\,4.6]{ramos1996equipartition} or any other result in his article.  

\begin{table}[ht] 
\begin{center}
\begin{tabular}{l c l l l }
  $\Delta(2^m,2)$ & $\le$ & $3 \cdot 2^m/2$& \\
  $\Delta(2^m,3)$ & $\le$ & $5 \cdot 2^m/2$& \\
 $\Delta(2^m,4)$ & $\le$ & $9 \cdot 2^m/2$& &\\
 $\Delta(2^m,5)$ & $\le$ & $15 \cdot 2^m/2$& & \\
\end{tabular}
\end{center}
\caption{Table taken from \cite[page 164]{ramos1996equipartition}. Here $m\geq 0$.}
\label{fig_ramos_upper_bounds}
\end{table}

In order to clarify Ramos' approach, we will describe his initial configuration space, which he modifies twice. The second modification is the basis for \cite[Lem.\,6.2]{ramos1996equipartition}. 
Given a dimension $d \geq 1$ and a number of hyperplanes $k \geq 1$ and masses $\mu_1, \dots, \mu_j$, the initial configuration space is defined as 
\[
B^{d-1} \times \dots \times B^{d-1} = B^{k(d-1)}.
\]
Here $B^{d-1}$ is regarded as the upper hemisphere of $S^{d-1}$, where each sphere $S^{d-1}$ is the space of all normal vectors to hyperplanes in $\R^{d}$ that bisect the first mass $\mu_1$. 
It is implicitly assume that the mass $\mu_1$ has a unique bisector in each direction.
In order for his first result \cite[Thm.\,4.6]{ramos1996equipartition} to hold, 
Ramos makes a first restricition to the configuration space:
\[
B^{n_1} \times \dots \times B^{n_k} \subseteq B^{k(d-1)},
\]
where $n_i \le d-1$ for all $i=1, \dots,k$ and $\sum n_i = (2^k-1)j-k$ \cite[Sec.\,4]{ramos1996equipartition}. 
Note that \cite[Thm.\,4.6]{ramos1996equipartition} does not yield the upper bounds in 
Table~\ref{fig_ramos_upper_bounds}. 

Let $\mu_1, \dots, \mu_j$ be masses on $\R^d$. 
For $x=(x_1, \dots, x_k) \in B^{n_1} \times \dots \times B^{n_k}$ and $i \in [k]$, let $H(x_i)$ be the unique hyperplane in $\R^d$ with normal vector $x_i$ that bisects the first mass $\mu_1$, where we regard the $x_i$ in $\R^d$ via the inclusions $B^{n_i} \subseteq B^{d-1} \lhook\joinrel\longrightarrow S^{d-1}$.  
For $\alpha \in \{0,1\}$, let $H^\alpha(x_i)$ be the positive (if $\alpha=0$) respectively negative (if $\alpha= 1$) closed half-space defined by $H(x_i)$. Observe the difference in notation to $H_{x_i}$, which we used to denote the affine hyperplane cooresponding to a point $x_i$ in the sphere $S^{d}$ of one dimension higher.

Ramos defines the  test map
\[
\xymatrix@1{
\Phi \colon B^{n_1} \times \dots \times B^{n_k}\ \ar[r]^-{\psi} & \ (\R^{2^k})^{\oplus j} \ \ar[r]^-{U\oplus \dots \oplus U} &  \ (\R^{2^k})^{\oplus j} \ \ar[r]^-{\pi} & \ (\R^{2^k-1})^{\oplus j}=U_k^{\oplus j},
}
\]
by
\[
\xymatrix@1{
(x_1, \dots, x_k) \ \ar@{|->}[r]^-{\psi}  & \ \Big(\mu_1\big(\bigcap_{i=1}^k H^{\alpha_i}(x_1)\big),\dots, \mu_j\big(\bigcap_{i=1}^k H^{\alpha_i}(x_k)\big)\Big)_{(\alpha_1,\dots,\alpha_k) \in (\Z/2)^k}.
}
\]
The map $\psi$ is followed by an orthogonal coordinate transformation $U \oplus \dots \oplus U=U^{\oplus j}$ given by the matrix
\[
U = \big( \epsilon_{i_1,\dots, i_k}^{j_1, \dots, j_k} \big) \quad \text{for} \quad (i_1,\dots,i_k),(j_1,\dots,j_k) \in (\Z/2)^k,
\]
where
\[
\epsilon_{i_1,\dots, i_k}^{j_1, \dots, j_k} = (-1)^b \quad \text{and} \quad b = (i_1,\dots,i_k)^t(j_1,\dots,j_k).
\]
The map $\pi$ chops off the coordinates of $(U \oplus \dots \oplus U) \circ \phi$ corresponding to the row of $U$ with index $(i_1, \dots, i_k) = (0,\dots,0)$. In these coordinates,  $(U \oplus \dots \oplus U) \circ \phi$ is constant and equal to $1$, since the value of such a coordinate is the sum, for a fixed mass, of the masses of all of the orthants. The map $\Phi$ can be viewed as a map to a $((2^k -1)j -k)$-dimensional subspace of $(\R^{2^k-1})^{\oplus j} \cong U_k^{\oplus j}$ since the map $\Phi$ has $k$ zero-components due to the fact that all hyperplanes bisect the first mass by definition. 

\begin{proposition}[{\cite[Property\,4.4]{ramos1996equipartition}}]
Let $x=(x_1,\dots,x_k) \in \R^d$, then $\Phi(x) =0$ if and only if the hyperplanes $H(x_i)
\subset \R^d$ with normal vectors $x_i$ that bisect the first mass~$\mu_1$  form an equipartition of the masses $\mu_1, \dots, \mu_j$. Moreover, if $\Phi(x)= 0$, then $\Delta(j,k) \le d$.
\end{proposition}

In the following definition, Ramos introduces the notion of a map that is equivariant on the boundary of the domain and calls this \emph{antipodal}. For this we let~$(\Z/2)^k$ act antipodally on the boundary of $B^{n_1} \times \dots \times B^{n_k}$.

\begin{definition}[{\cite[p.\,151]{ramos1996equipartition}}]
A continuous map $f\colon B^{n_1} \times \dots \times B^{n_k} \longrightarrow \R^{(2^k-1)j}$ is \emph{antipodal in the $m$-th component with respect to the $n$-th ball with antipodality $a_{pq} \in \{0,1\}$}, for $q \in [k]$ and $p \in [(2^k-1)j]$, if
\begin{multline*}
f(x_1,\dots, -x_q,\dots, x_k) = (-1)^{a_{pq}} f_p(x_1,\dots, x_q,\dots, x_k)  \\ \text{ for all } (x_1, \dots, x_k) \in B^{n_1} \times \dots \times S^{n_q -1} \times \dots \times B^{n_k}.
\end{multline*}
Call $f$ \emph{antipodal} if $f$ is antipodal in all components with respect to all balls. In this case we call $A = (a_{pq})_{p,q} \; \in \R^{(2^k-1)j \times k}$ the \emph{antipodality matrix} of $f$. 
\end{definition}

Using the antipodality matrix $A$, we define an action of $(\Z/2)^k$ on $\R^{(2^k-1)j}$ by letting the generators of $(\Z/2)^k$ act by changing the signs of vectors in $\R^{(2^k-1)j}$ according to the columns of $A$. In this restricted sense, $f$ is equivariant on the boundary of $B^{n_1} \times \dots \times B^{n_k}$.

\begin{proposition}[{\cite[Property 4.3]{ramos1996equipartition}}] \label{prop_antipod_matrix_of_test_map}
The test map $\Phi$ is antipodal. Its antipodality in the component with index $(i_1,\dots, i_q, \dots, i_k)$ with respect to the $q$-th ball is $i_q$ (for any mass). Hence, the rows of $A$ are precisely all $0/1$-vectors of length $k$, each repeated $j$ times, up to some re-ordering of the rows that depends only the labeling of the components of $\Phi$. If we regard $\Phi$ as mapping into  $U_k^{\oplus j}$, then $A$ consists of all $0/1$-vectors of length $k$ except of  $(0,\dots,0)$, each repeated $j$ times.
\end{proposition}

Ramos's method of proof is to show that the parity of the number of zeros of the test map $\Phi$ on the given domain is odd and hence the map has at least one zero.  In \cite[Thm.\,4.6]{ramos1996equipartition} he shows that if the permament of a certain matrix is odd, then the parity of the number of zeros of~$\Phi$ is also odd. However, this permanent is odd in only a few cases and in particular in none of the cases listed in Table~\ref{fig_ramos_upper_bounds}. To prove \cite[Lem.\,6.2]{ramos1996equipartition} and obtain the results in Table~\ref{fig_ramos_upper_bounds}, Ramos restricts the configuration space a second time with the goal of obtaining more cases where the matrix permanent is odd. Instead of a product of balls, he uses a subspace of a product of balls: For $p,q \geq 1$ define
\[
(B^{p})^{q}_\le = \{(x_1,\dots,x_{t}) \in (\R^{p})^q \; \colon \; \|x_1\| \le \|x_2\| \le \dots \le \|x_{q}\| \le 1 \} \subseteq (B^{p})^{q}.
\]
The space $(B^{p})^{q}_\le$ is a fundamental domain for the action of the symmetric group $\Sym_q$ on $(B^{p})^{q}$ given by permuting copies. \cite[Lem.\,6.2]{ramos1996equipartition} is a result that relates the parity of the number of zeros of the test map $\Phi$ on $(B^{p})^{q}_\le$ to the parity of the number of zeros of $\Phi$ on the boundary of $(B^{p})^{q}_\le$. 
Ramos parametrizes the boundary as follows: For $1 \le m < n \le q$, define sets
\begin{align*}
C_{m,n} &= \{(x_1, \dots, x_q) \in (B^p)^q_\le \; : \; \|x_m\| = \|x_n\| \},  \\
C_{q,q+1} &= \{(x_1, \dots, x_q) \in (B^p)^q_\le \; : \; \|x_q\|  = 1 \}.
\end{align*}
Here $C_{q,q+1}$ can be regarded as the ``lid'' of $(B^p)^q_\le$, where the ``top lid'' $X_{q,q+1}^+ = X_{q,q+1} \cap \{x_q \geq 0\}$ and the ``bottom lid'' $X_{q,q+1}^- = X_{q,q+1} \cap \{x_q \le 0\}$ are homeomorphic to $(B^p)^{q-1}_\le \times B^{p-1}$. Hence 
\[
\mathrm{bd}(B^p)^q_\le \; \; = \biguplus\limits_{1\le m \le q} C_{m,m+1}  \; ,
\]
where ``$\biguplus$'' denotes the union of sets whose relative interiors are disjoint. On the sets $C_{m,n}$, Ramos defines a permutation action given by
\begin{align*}
\beta_{mn} \colon C_{m,n}  &\longrightarrow C_{m,n}\\
(x_1, \dots, x_m, \dots, x_n, \dots, x_q) &\longmapsto (x_1, \dots, x_n, \dots, x_m, \dots, x_q).
\end{align*}
Notice that points in the subsets $\{ x \in (B^{p})^{q}_\le \; :  \; x_m = x_n \} \subset C_{m,n}$ are fixed by this action. Hence the action is not fixed point free.

For the proofs, Ramos switches to a piecewise-linear (PL) approximation of the test map that maps the simplices of a ``symmetric'' triangulation of $(B^{p})^{q}_\le$ into general position with respect to the origin. See the following definition for these notions.

\begin{definition}[{\cite[p.\,149]{ramos1996equipartition}}]
If $T$ is a pseudomanifold, then we call a map $r \colon \|T\| \longrightarrow \R^n$ \emph{piecewise linear} if it is affine on every simplex of $T$. We call $r$ \emph{non-degenerate} if given any $m$-simplex $\sigma \in T$, any $m$ component functions of $r$ have at most one common zero on $\sigma$ and any common zero lies in the relative interior of $\sigma$. We will say that $r$ is \emph{NDPL} if $r$ is both non-degenerate and piecewise linear. 
\end{definition}

The test map or its NDPL approximation is again required to be ``equivariant'' in some sense. This is made precise in the following definition.

\begin{definition}[{\cite[p.\,162]{ramos1996equipartition}}]
Given a map $r=(r',r'')\colon (B^p)^q_\le \longrightarrow \R^{pq}$, where $r'$ denotes the first $pq-1$ components of $r$ and $r''$ the last component, we call $r$ \emph{symmetric for the zeros in the boundary} if for all $1 \le m < n \le q$ and all $x \in C_{m,n}$ the following implication holds:
\[
r'(x) = 0 \; \text{ implies that } \; r'(\beta_{mn} (x)) = 0 \text{ and } r''(x) = r''(\beta_{mn}(x)).
\]
\end{definition}

\begin{lemma}[{\cite[Lem.\,6.2]{ramos1996equipartition}}] \label{lem_lem6_2}
Let $r=(r',r'') \colon (B^p)^q_\le \longrightarrow \R^{pq}$ be a map where $r'$ denotes the first $pq-1$ components and $r''$ the last component. Suppose $r$ is NDPL and symmetric for the zeros in the boundary. Let $r$ be antipodal in the last component with respect to the $q$-th ball and let $a=a_{pq,q} \in \{0,1\}$ be its antipodality. If $P(r',r''; (B^p)^q_\le)$ denotes the parity of the number of zeros of $r$ in $(B^p)^q_\le$ and $P(r'; (B^p)^{q-1}_\le \times B^{p-1})$ denotes the parity of the number of zeros of $r'$ in $X_{q,q+1}^+ \approx (B^p)^{q-1}_\le \times B^{p-1}$, the ``top lid'' of the boundary of $(B^p)^q_\le$, then we have the following equality:
\[
P(r',r''; (B^p)^q_\le) = a \cdot P(r'; (B^p)^{q-1}_\le \times B^{p-1}).
\]
\end{lemma}

\begin{example}[Counterexample to {\cite[Lem.\,6.2]{ramos1996equipartition}}] 
\label{ex_counterexample_lemma_6_2} 
\normalfont
This example exploits the simple fact that the permutation action on the coordinates in $C_{m,n}$ has fixed points, a fact that Ramos does not account for in his proof of \cite[Lem.\,6.2]{ramos1996equipartition}.
Let $p =1$ and $q =3$. Then
\[
(B^p)^q_\le = (B^1)^3_\le = \{ (x,y,z) \in \R^3 : |x| \le |y| \le |z| \le 1\}.
\]
See Figures \ref{fig_pic_counterexample_view_one} and \ref{fig_pic_counterexample_view_two} for a visualization of $(B^1)^3_\le$. Define the following sets and color them as in the Figures:
\begin{align*}
F_{x,y} &= \{(x,y,z) \in (B^1)^3_\le \; : \; x = y \} \subset C_{1,2}, \quad \,\,\text{~~~~~~``red''}  \\
F_{y,z} &= \{(x,y,z) \in (B^1)^3_\le \; : \; y = z \} \subset C_{2,3}, \quad \,\,\text{~~~~~~``blue''} \\
F_{x,z} &= \{(x,y,z) \in (B^1)^3_\le \; : \; x = z \} \subset C_{1,3}, \quad \,\,\text{~~~~~~``green''} \\
\mathrm{Top} &= \{(x,y,z) \in (B^1)^3_\le \; : \;  z = 1 \} \subset C_{3,4},\quad \,\,\text{~~~~~~``black''}\\
\mathrm{Bot}  &= \{(x,y,z) \in (B^1)^3_\le \; : \;  z = -1 \} \subset C_{3,4}. \quad\text{``black''}
\end{align*}

We will now construct a map $r=(r',r'')\colon (B^1)^3_\le \longrightarrow \R^{3}$  that contradicts \cite[Lem.\,6.2]{ramos1996equipartition}.
\begin{compactenum}[\normalfont (i)]
\item Rotate $(B^1)^3_\le$ by $90^\circ$ to the right along the $y$-axis. Now $\mathrm{Top}$ and $\mathrm{Bot}$ lie in the two parallel hyperplanes $\{ x=1 \}$ and $\{x =-1\}$.
\item Rotate $(B^1)^3_\le$ along the $x$-axis and translate it such that the $z$-axis runs through $F_{x,y,1}$ and $F_{y,z,1}$ and the origin lies in the interior of the tetrahedron that has $F_{x,y,1}$ and $F_{y,z,1}$ as two of its faces. See Figure \ref{fig_pic_counterexample_view_three}.
\end{compactenum}
\end{example}

Since $r''(x,y,-z) = r''(x,y,z) = (-1)^0 r''(x,y,z)$, the map $r = (r', r'')$ is antipodal in the last component with respect to the third ball with antipodality $a=0$. It is easy to check that $r$ is non-degenerate. Moreover, $r$ has exactly one zero in $(B^1)^3_\le$. Hence
\[
P(r',r''; (B^1)^3_\le) = 1 \neq 0 = 0 \cdot P(r'; (B^1)^{2}_\le \times B^{0}).
\]

\begin{figure}[!ht] 
\begin{subfigure}[b]{0.48\textwidth}
\includegraphics[width=1\textwidth]{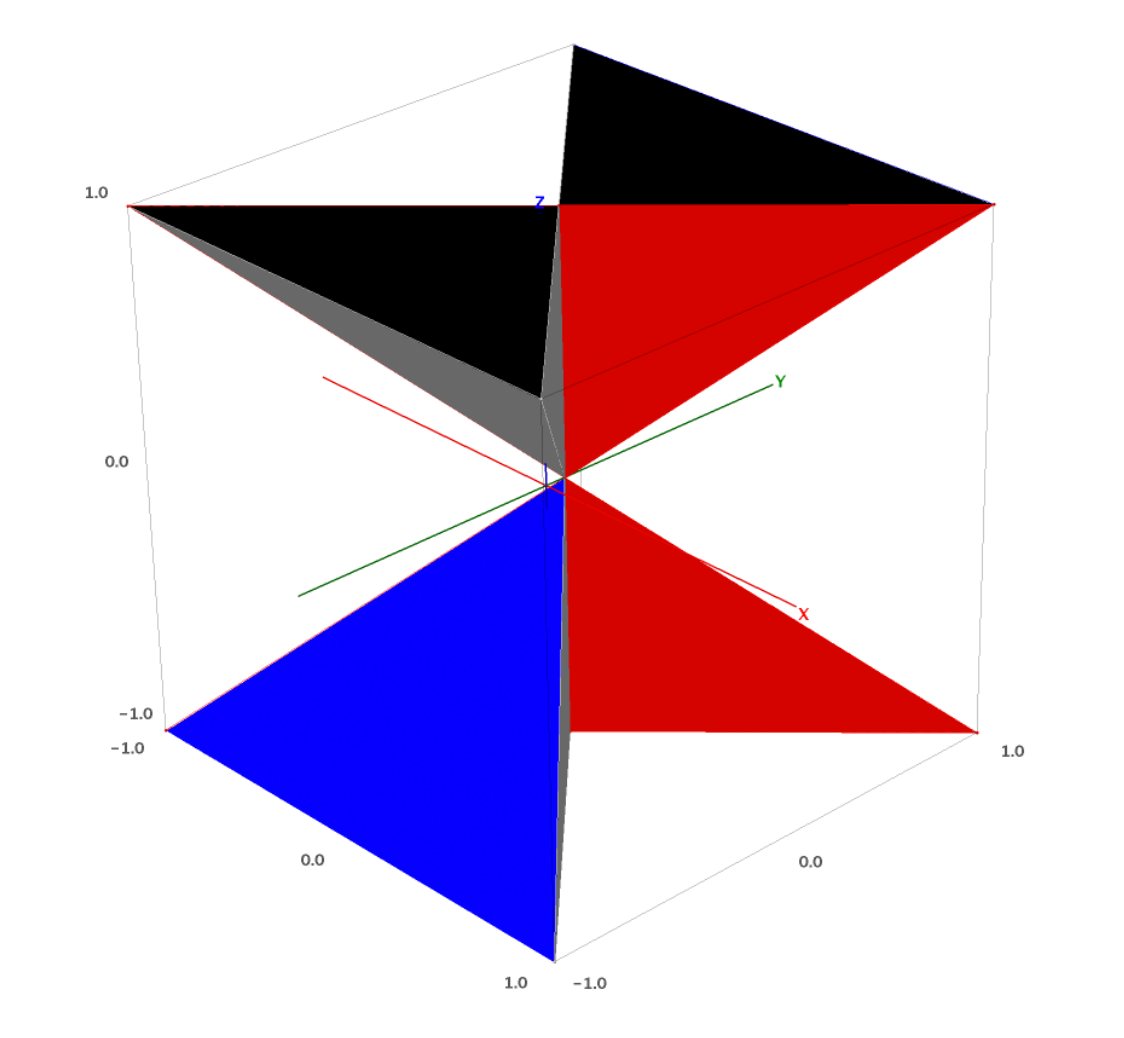}
\caption{One view of $(B^1)^3_\le$} \label{fig_pic_counterexample_view_one}
\end{subfigure}
\begin{subfigure}[b]{0.48\textwidth}
\includegraphics[width=1\textwidth]{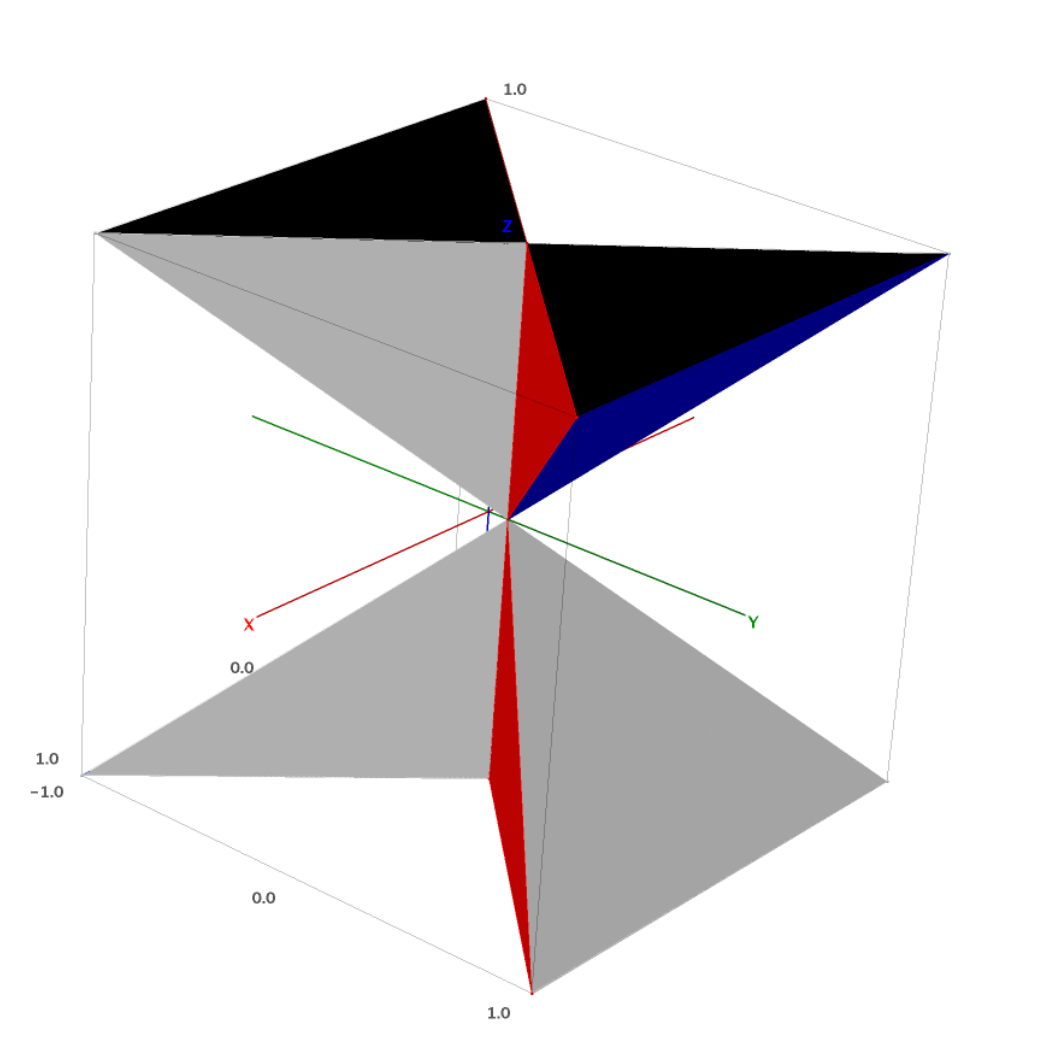}
\caption{Another view of $(B^1)^3_\le$}\label{fig_pic_counterexample_view_two}
\end{subfigure}
\begin{subfigure}[b]{0.48\textwidth}
\includegraphics[width=0.9\textwidth]{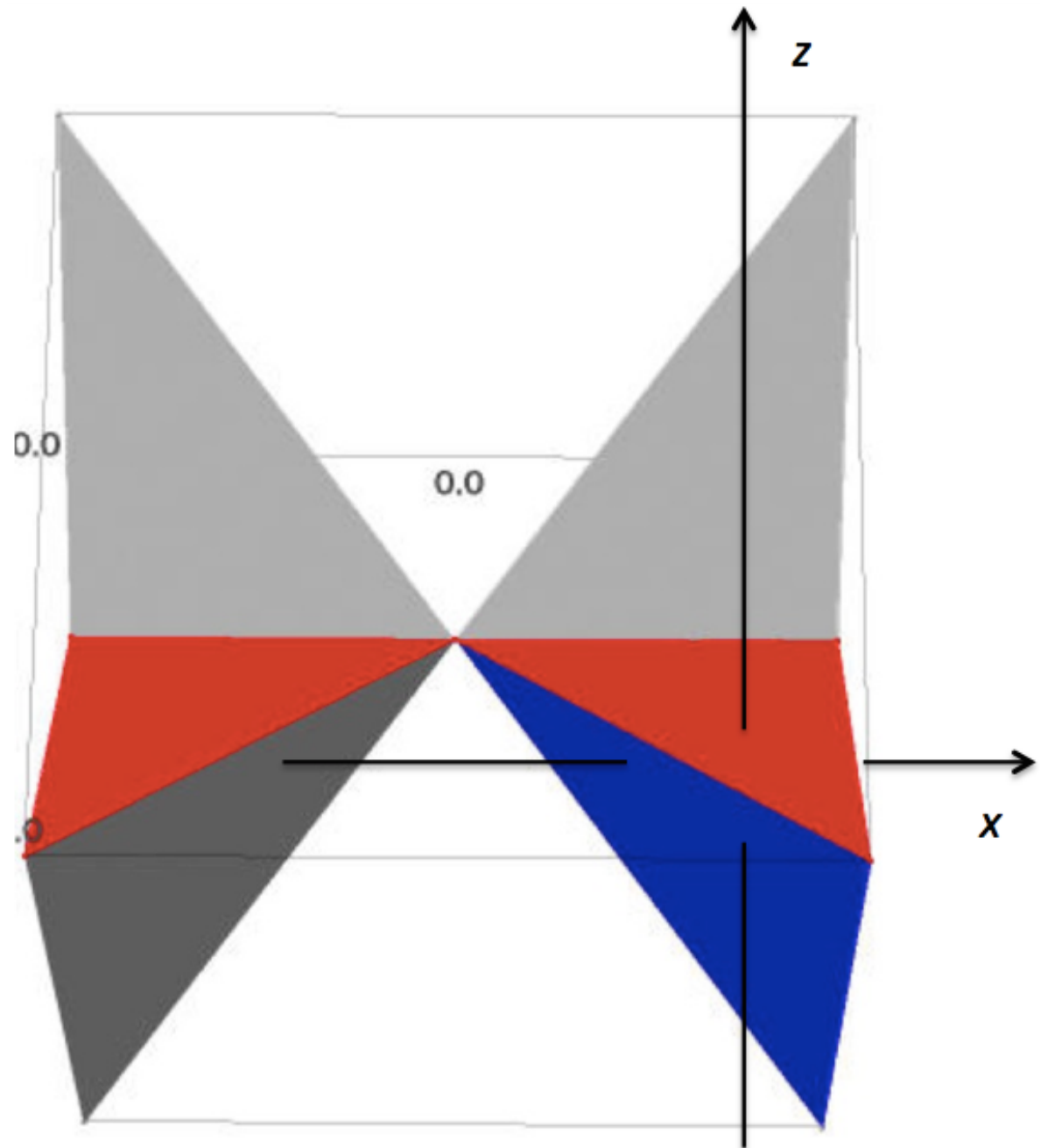}
\caption{The $r$ image of $(B^1)^3_\le$}\label{fig_pic_counterexample_view_three}
\end{subfigure}
\caption{}\label{fig_pic_counterexample}
\end{figure}

\section{Further gaps in the literature}
\label{sec:further_gaps}

In this section we explain essential gaps in proofs of the main results in the papers of Mani-Levitska et al.~\cite{mani2006} and \v{Z}ivaljevi\'c~\cite{zivaljevic2011equipartitions}.

\subsection{Gaps in~\cite{mani2006}}
Mani-Levitska et al.\ in their 2006 paper \cite{mani2006} studied the Ramos conjecture in the case of two hyperplanes, $k=2$.
One of the main result of this paper \cite[Thm.\,4]{mani2006}
was a criterion under which for special values
of $m$, in particular for $m=1$, one would get $\Delta(4m+1,2)\le 6m+2$.

To get this criterion, they used
the product configuration space/test map scheme and applied the 
equivariant obstruction theory of tom Dieck \cite[Sec.\,II.3]{tom1987transformation} 
in order to study the non-existence of $D_8$-equivariant maps 
$S^d\times S^d\longrightarrow S(U_2^{\oplus j})$.
Indeed, in the beginning of \cite[Sec.\,2.3.3]{mani2006} the authors recall details on the equivariant obstruction theory they apply as well as about the first isomorphism that will be used in the identification of the obstruction element:

\begin{quote}
{\small
\cite[\bf Section\,2.3.3]{mani2006} \newline Once a problem is reduced to the question of (non) existence of equivariant map, one can use some standard topological tools for its solution. For example, one can use the cohomological index theory for this purpose [14,17,45,47]. This approach is discussed in Section 4.1. In this paper our main tool is elementary equivariant obstruction theory [13], refined by some basic equivariant bordism, and group homology calculations.

Suppose that $M^n$ is orientable, $n$-dimensional, {\bf free} $G$-manifold and that $V$ is a $m$-dimensional, real representation of $G$. Then the first obstruction for the existence of an equivariant map $f\colon M \longrightarrow S(V)$, is a cohomology class
\[
\omega\in H^{m}_G\big(M,\pi_{m-1}\big(S(V)\big)\big)
\]
in the appropriate equivariant cohomology group [13, Section II.3], where $\pi_k(S(V))$ is seen as a $G$-module. 
The action of $G$ on $M$ induces a $G$-module structure on the group $H_n(M,\Z)\cong\Z$ which is denoted by $\mathcal{O}$. 
The associated homomorphism $\theta\colon G\longrightarrow\{-1,+1\}$
is called the orientation character. Let $A$ be a (left) $G$-module. 
The Poincaré duality for equivariant (co)homology is the following isomorphism [39],
\[
H^k_G(M,A)\overset{D}{\longrightarrow} H^G_{n-k}(M,A\otimes\mathcal{O}).
\]}
\end{quote}

\noindent
(Boldface added for emphasis.)
In \cite[Sec.\,2.6]{mani2006} they present further isomorphisms that will be used in the identification of the obstruction element:

\begin{quote}
{\small
\cite[\bf Section\,2.6]{mani2006}\newline
By equivariant Poincar\'e duality, Section 2.3.3, the dual $D(\omega)$ of the first obstruction cohomology class $\omega\in H^m_G(M,\pi_{m-1}S(V))$ lies in the equivariant homology group $H^G_{n-m}(M,\pi_{m-1}S(V)\otimes\mathcal{Z})$.
If $M$ is $(n-m)$-connected, then there is an isomorphism [11, Theorem II.5.2]
\[
H^G_{n-m}(M,\pi_{m-1}S(V)\otimes\mathcal{Z}) \overset{\cong}{\longrightarrow} H_{n-m}(G,\pi_{m-1}S(V)\otimes\mathcal{Z}).
\]
This allows us to interpret $D(\omega)$ as an element in the latter group. 
Moreover, if the coefficient $G$-module $\pi_{m-1}S(V)\otimes\mathcal{Z}$ is trivial, then the homology group $H_{n-m}(G,\Z)\cong H_{n-m}(BG,\Z)$ is for $n-m\leq 3$ isomorphic to the oriented $G$-bordism group $\Omega_{n-m}(G)\cong\Omega_{n-m}(BG)$, i.e. to the groups based on free, oriented $G$-manifolds [12].

Our objective is to identify the relevant obstruction classes. 
Already the algebraically trivial case $H_0(G,M)\cong M_G$, where $M_G = \Z\otimes M$ is the group of coinvariants, may be combinatorially sufficiently interesting. 
Indeed, the parity count formulas applied in [32], see also [49, Section 14.3], may be seen as an instance of the case $M_G\cong\Z/2$.

However, the most interesting examples explored in this paper involve the identification of $1$-dimensional obstruction classes. 
Since these classes in practice usually arise as the fundamental classes of zero set manifolds, our first choice will be the bordism group $\Omega_1(G)$.
}
\end{quote}

After presenting the method used in the paper \cite{mani2006} for the study of the non-existence of $D_8$-equivariant maps $S^d\times S^d\longrightarrow S(U_2^{\oplus j})$ we can point out the gap.
For the method to work the action of the group (in this case $D_8$) on the manifold (in this case $S^d\times S^d$) has to be free.
The action of $D_8$ on $S^d\times S^d$ is {\bf not free} and therefore the method can not be applied to the problem of the non-existence of $D_8$-equivariant maps $S^d\times S^d\longrightarrow S(U_2^{\oplus j})$.
Consequently, all the claims by Mani-Levitska et al. derived from the application of this method --- namely \cite[Thm.\,4,\,Prop.\,25,\,Thm.\,33,\,Cor.\,37]{mani2006} --- are not proven. 
Furthermore, we point out that 
\begin{compactitem}[~~$\bullet$]
\item the Poincar\'e duality isomorphism 
$H^k_G(M,A)\overset{D}{\longrightarrow} H^G_{n-k}(M,A\otimes\mathcal{O})$
stands only with the assumption that $M$ is an oriented compact manifold with a free $G$-action; a complete proof can be found in \cite[Thm.\,1.4]{dimitrijevic},
\item the isomorphism $H^G_{n-m}(M,A) \overset{\cong}{\longrightarrow} H_{n-m}(G,A)$ holds for a trivial $G$-module $A$ when $M$ is an $(n-m)$-connected  space on which the $G$-action is free.
\end{compactitem}

\smallskip\noindent
Finally, let us mention that already in 1998 \v{Z}ivaljevi\'c \cite[Proof of Prop.\,4.9]{zivaljevic1998users-guide-2} has given a suggestion how to 
deal with the presence of non-free actions in the context of equivariant obstruction theory  applied to the Ramos conjecture:
There he studied the non-existence of a $(\Z/2\oplus D_8)$-equivariant map $(S^3)^3\longrightarrow S(\R^9)$ with non-free action on the domain
using relative equivariant obstruction theory.  

\subsection{A gap in \cite{zivaljevic2011equipartitions}~\cite{zivaljevic2011equipartitions-published}}
In his 2011 paper \cite{zivaljevic2011equipartitions} that was published in April 2015 \cite{zivaljevic2011equipartitions-published}, \v{Z}ivaljevi\'c  studied the Ramos conjecture in the case of two hyperplanes, $k=2$.
The main result \cite[Thm.\,2.1]{zivaljevic2011equipartitions}~\cite[Thm.\,2.1]{zivaljevic2011equipartitions-published} claims that $\Delta(4\cdot 2^k+1,2)=6\cdot 2^k+2$.
For this claim we gave a degree-based proof, see Theorem~\ref{thm:Delta(2^t+1,2)}. 

In order to study the non-existence of $D_8$-equivariant maps induced by the product configuration scheme $S^d\times S^d\longrightarrow S(U_2^{\oplus j})$ \v{Z}ivaljevi\'c in \cite[Sec.\,12]{zivaljevic2011equipartitions}~\cite[App.\,B.2]{zivaljevic2011equipartitions-published} introduces an 
``algebraic equivariant obstruction theory.''
We explain why the proofs for \cite[Thms.\,2.1 and 2.3]{zivaljevic2011equipartitions}~\cite[Thms.\,2.1 and 2.2]{zivaljevic2011equipartitions-published}
using this obstruction theory are not complete, as they fail to validate essential
preconditions that are not automatically provided by this theory.

Following \cite[Sec.\,12]{zivaljevic2011equipartitions}~\cite[App.\,B.2]{zivaljevic2011equipartitions-published}, suppose that $X$ is a $d$-dimensional $G$-space with admissible filtration \cite[Def.\,12.5]{zivaljevic2011equipartitions}~\cite[Def.\,B.3]{zivaljevic2011equipartitions-published}:
\[
\emptyset=X_{-1}\subset X_0\subset X_1\subset \cdots\subset X_{n-1}\subset X_n\subset X_{n+1}\subset \cdots\subset X_d=X.
\]
Furthermore, let $Y$ be a $G$-CW-complex with associated filtration by skeleta:
\[
\emptyset=Y_{-1}\subset Y_0\subset Y_1\subset \cdots\subset Y_{n-1}\subset Y_n\subset Y_{n+1}\subset \cdots\subset Y_{\nu}=Y.
\]
Then, according to \cite[Prop.\,12.11]{zivaljevic2011equipartitions}~\cite[Prop.\,B.6]{zivaljevic2011equipartitions-published}, if we assume that there exists a $G$-equivariant map $f\colon X\longrightarrow Y$, then there exists a chain map 
\[
f_*\colon H_n(X_n,X_{n-1};\Z)\longrightarrow H_n(Y_n,Y_{n-1};\Z)
\]
between the associated augmented chain complexes of $\Z[G]$-modules:
\[
\xymatrix@!C=1.8em{
\ldots\ar[r]^{\partial} &C_{n+1}\ar[r]^{\partial}\ar[d]^{f_{n+1}}&C_n\ar[r]^{\partial}\ar[d]^{f_n}&C_{n-1}\ar[r]^{\partial}\ar[d]^{f_{n-1}} &\ldots\ar[r]^{\partial}&C_1\ar[r]^{\partial}\ar[d]^{f_1}&C_0\ar[r]^{\partial}\ar[d]^{f_0}&\Z\ar[r]\ar[d]^{=}&0\\
\ldots\ar[r]^{\partial} &D_{n+1}\ar[r]^{\partial}&D_n\ar[r]^{\partial}&D_{n-1}\ar[r]^{\partial} &\ldots\ar[r]^{\partial}&D_1\ar[r]^{\partial}&D_0\ar[r]^{\partial}&\Z\ar[r]&0
}
\]
where $C_n=H_n(X_n,X_{n-1};\Z)$ and $D_n=H_n(Y_n,Y_{n-1};\Z)$ for every $n$.

\noindent
Now \cite[Sec.\,12.3]{zivaljevic2011equipartitions}~\cite[App.\,B.3]{zivaljevic2011equipartitions-published} studies  the existence of chain maps between chain complexes of $\Z[G]$-modules. 
\cite[Prop.\,12.13]{zivaljevic2011equipartitions}~\cite[Prop.\,B.7]{zivaljevic2011equipartitions-published} introduces an obstruction theory as follows: 
For $n+1\leq d$,
\begin{compactitem}[~~$\bullet$]
\item a finite chain complexes of $\Z[G]$-modules $C_*=\{C_k\}_{k=-1}^d$ and $D_*=\{D_k\}_{k=-1}^d$, with $C_{-1}=D_{-1}=\Z$, and
\item a fixed partial chain map $F_{n-1}=(f_j)_{j=-1}^{n-1}\colon \{C_k\}_{k=-1}^{n-1}\longrightarrow\{D_k\}_{k=-1}^{n-1}$,
\end{compactitem}
assume that $F_{n-1}$ can be extended to dimension $n$, i.e., there exists $f_n\colon C_n\longrightarrow D_n$ such that $\partial f_n=f_{n-1}\partial$. 
Then \cite[(12.14)]{zivaljevic2011equipartitions}~\cite[(B.7)]{zivaljevic2011equipartitions-published} defines the obstruction to the existence of a partial chain map 
\[
F_{n+1}=(f_j)_{j=-1}^{n+1}\colon \{C_k\}_{k=-1}^{n+1}\longrightarrow\{D_k\}_{k=-1}^{n+1},
\]
which extends the partial chain map $F_{n-1}$, with a possible modification of $f_n$, as an appropriate element $\theta$ of the cohomology group:
\[
H^{n+1}(C_*;H_n(D_*))=H_{n+1}(\mathrm{Hom}(C_*,H_n(D_*)).
\]
The element $\theta$ is represented by the cocycle \cite[(12.15)]{zivaljevic2011equipartitions}~\cite[(B.8)]{zivaljevic2011equipartitions-published}:
\[
\theta(f_n)\colon C_{n+1}\overset{\partial}{\longrightarrow}C_n\overset{f_n}{\longrightarrow}Z_n(D_*)\overset{\pi}{\longrightarrow}H_n(D_*).
\]
Now \cite[Prop.\,12.13]{zivaljevic2011equipartitions}~\cite[Prop.\,B.7]{zivaljevic2011equipartitions-published} states that vanishing of $\theta$ is not only necessary but also sufficient for the existence of the extension $F_{n+1}$ if $C_n$ and $C_{n+1}$ are projective modules.

The obstruction $\theta$ highly depends on the partial chain map $F_{n-1}=(f_j)_{j=-1}^{n-1}$.
The first paragraph of \cite[Sec.\,12.4]{zivaljevic2011equipartitions}~\cite[Sec.\,B.4]{zivaljevic2011equipartitions-published} comments on this issue  as follows:

\begin{quote}
{\small
\cite[\bf 12.4., B.4., Heuristics for evaluating the obstruction $\theta$.]{zivaljevic2011equipartitions,zivaljevic2011equipartitions-published}\newline
In many cases the chain map $F_{n-1}=(f_j)_{j=-1}^{n-1}$, which in Proposition 12.13 serves as an input for calculating
the obstruction $\theta$, is unique up to a chain homotopy. 
This happens for example when $D_*$ is a chain complex associated to a $G$-sphere $Y$ of dimension $n$.
}
\end{quote}

The last sentence is not true:
In order to guarantee that the input partial chain map $F_{n-1}=(f_j)_{j=-1}^{n-1}$ is unique up to a chain homotopy an additional condition on the chain complex $C_*$ needs to be fulfilled, for example that $\{C_k\}_{k=-1}^{n-1}$ is a sequence of projective $\Z[G]$-modules. 

The algebraic obstruction theory just described is applied in \cite{zivaljevic2011equipartitions}~\cite{zivaljevic2011equipartitions-published}  to the problem of the non-existence of a $D_8$-equivariant map 
$S^d\times S^d\longrightarrow S(U_2^{\oplus j})$:
\begin{compactitem}[~~$\bullet$]
\item in \cite[Sec.\,6]{zivaljevic2011equipartitions}~\cite[Sec.\,6]{zivaljevic2011equipartitions-published} an admissible filtration of $S^d\times S^d$ is defined,
\item in \cite[Sec.\,7]{zivaljevic2011equipartitions}~\cite[Sec.\,7]{zivaljevic2011equipartitions-published} the top three levels of the associated chain complex $C_*$ of $S^d\times S^d$ are described as projective $\Z[D_8]$-modules,
\item in \cite[Prop.\,9.21]{zivaljevic2011equipartitions}~\cite[Prop.\,9.9]{zivaljevic2011equipartitions-published} evaluates the obstruction $\theta$ for particular input data $F_{2d-2}=(f_j)_{j=-1}^{2n-2}$ proving that it does not vanish.
\end{compactitem}
Since the $D_8$-action on $S^d\times S^d$ is not free the chain complex $C_*$ of $\Z[D_8]$-modules associated to $S^d\times S^d$ is not guaranteed to be
a chain complex of projective $\Z[D_8]$-modules.
Thus different input data $F_{2d-2}=(f_j)_{j=-1}^{2n-2}$ need not define the same obstruction $\theta$ computed in \cite[Prop.\,9.21]{zivaljevic2011equipartitions}~\cite[Prop.\,9.9]{zivaljevic2011equipartitions-published}.  
Consequently, no conclusion about the non-existence of an extension of $F_{2d-1}$, and further of a $D_8$-equivariant map $S^d\times S^d\longrightarrow S(U_2^{\oplus j})$, can be obtained from computation of just one obstruction.   
This exhibits an essential gap in the proof of the main result \cite[Thm.\,2.1]{zivaljevic2011equipartitions}~\cite[Thm.\,2.1]{zivaljevic2011equipartitions-published} as well as a serious deficiency in the proposed algebraic obstruction theory.
 
\subsubsection*{Acknowledgements.}
We are grateful to John M. Sullivan for very good advice and useful comments.

\end{document}